\documentclass[11pt]{amsart}
\usepackage{amssymb,amsmath,amsthm,amsrefs}
\usepackage{bm}
\usepackage{fullpage}
\usepackage{enumerate}
\usepackage{verbatim}

\BibSpec{collection.article}{
    +{}  {\PrintAuthors}                {author}
    +{,} { \textit}                     {title}
    +{.} { }                            {part}
    +{:} { \textit}                     {subtitle}
    +{,} { \PrintContributions}         {contribution}
    +{,} { \PrintConference}            {conference}
    +{}  {\PrintBook}                   {book}
    +{,} { }                            {booktitle}
    +{ } { }				{volume}
    +{,} { \PrintEditorsC}              {editor}
    +{,} { }                            {publisher}
    +{,} { \PrintDateB}                 {date}
    +{,} { pp.~}                        {pages}
    +{,} { }                            {status}
    +{,} { \PrintDOI}                   {doi}
    +{,} { available at \eprint}        {eprint}
    +{}  { \parenthesize}               {language}
    +{}  { \PrintTranslation}           {translation}
    +{;} { \PrintReprint}               {reprint}
    +{.} { }                            {note}
    +{.} {}                             {transition}
    +{}  {\SentenceSpace \PrintReviews} {review}
}

\newtheorem{theorem}[equation]{Theorem}
\newtheorem{lemma}[equation]{Lemma}
\newtheorem{prop}[equation]{Proposition}
\newtheorem{cor}[equation]{Corollary}

\numberwithin{equation}{section}

\newcommand{\mout}{m_{out}}

\newcommand{\Rbar}{\overline{R}}

\newcommand{\R}{\mathbb{R}}
\newcommand{\C}{\mathbb{C}}

\newcommand{\Met}{\mathop{Met}}

\newcommand{\madm}{m_{_{ADM}}}
\newcommand{\mH}{m_{_H}}

\newcommand{\Scal}{\mathcal{S}}
\newcommand{\Bcal}{\mathcal{B}}
\newcommand{\Hcal}{\mathcal{H}}

\newcommand{\Ucal}{\mathcal{U}}

\newcommand{\Qcal}{\mathcal{Q}}

\newcommand{\Sdot}{\dot{\Scal}}
\newcommand{\Bdot}{\dot{\Bcal}}

\newcommand{\etilde}{\widetilde{\eta}}

\newcommand{\Ric}{\operatorname{Ric}}

\newcommand{\Y}{Y_{\ell,m}}

\newcommand{\shc}[1]{\underset{\ell,m}{#1}}

\newcommand{\abs}[1]{\left\lvert#1\right\rvert}
\newcommand{\norm}[1]{\left\|#1\right\|}

\title{The Bartnik-Bray outer mass of small metric spheres in time-symmetric 3-slices}
\author{David Wiygul}
\address{Department of Mathematics, University of California, Irvine, CA 92697}
\email{dwiygul@uci.edu}

\begin{document}

\begin{abstract}
Given a sphere with Bartnik data close to that
of a round sphere in Euclidean $3$-space,
we compute its Bartnik-Bray outer mass
to first order in the data's deviation from the standard sphere.
The Hawking mass gives a well-known lower bound,
and an upper bound is obtained by estimating the mass of a static vacuum extension. 
As an application we confirm that
in a time-symmetric slice
concentric geodesic balls
shrinking to a point
have mass-to-volume ratio converging to the energy density at their center,
in accord with physical expectation
and the behavior of other quasilocal masses.
For balls shrinking to a point
where the Riemann curvature tensor vanishes
we can also compute
the outer mass to fifth order in the radius---the term is proportional
to the Laplacian of the scalar curvature at the center---but
our estimate is not refined enough to identify this term
in general at a point where merely the scalar curvature vanishes.
In particular it cannot discern gravitational
contributions to the mass.
\end{abstract}

\maketitle

\section{Introduction}

When studying the properties of a given definition for quasilocal mass,
it is natural to investigate the asymptotic behavior of the mass
on large and small spheres (or the regions they bound)
so as to assess
its relation to the canonical measures of mass in those regimes,
namely (for regions in time-symmetric slices)
the ADM energy in the large and the matter-field energy density in the small.
We have for example
\cite{HorSch}, \cite{BrownLauYork}, \cite{Yu}, \cite{FanShiTam}, and \cite{ChenWangYau},
concerning a variety of quasilocal masses on small spheres.
For the Bartnik mass (\cite{BarNew}, \cite{BarTHL}, \cite{BarICM}) in particular,
while Huisken and Ilmanen have addressed the large-sphere limit in \cite{HI},
there do not appear to be any asymptotic estimates for small spheres in the literature,
and indeed the question of the (suitably scaled)
small-sphere limit has been raised explicitly by Szabados
in the review \cite{Sza}.
Here we give a partial answer,
substituting Bray's outer-minimizing condition (\cite{Bray}) for Bartnik's no-horizon condition
and allowing extensions which are not smooth across the boundary.
In fact the precise quasilocal mass we consider is essentially one studied by Miao
(whose sign convention for mean curvature we caution is opposite ours below)
in Section 3.3 of \cite{MiaoVarEff}.

Specifically, let $\mathcal{PM}$ denote the set of complete Riemannian metrics
on $M:=\{\abs{x} \geq 1\} \subset \R^3$
having nonnegative scalar curvature and whose Cartesian components $g_{ij}$
are $C^2$ up to the boundary and satisfy the decay requirements
  \begin{equation}
  \label{afdecay}
    \abs{g_{ij}(x)-\delta_{ij}}+\abs{x}\abs{g_{ij,k}(x)}+\abs{x}^2\abs{g_{ij,k\ell}(x)}=O(\abs{x}^{-1}).
  \end{equation}

Given any metric $g$ on $M$ or on the closed unit ball $B$
and writing $\iota: S^2 \to \R^3$
for the inclusion map of the unit sphere $S^2=\partial M$ in $\R^3$,
we adopt the sign convention that the mean curvature $\Hcal_{\iota}[g]$
induced by $g$ on $\partial M$ is the divergence of the unit normal
directed toward the origin
(so that $\Hcal_{\iota}[\delta]=-2$ 
for the standard Euclidean metric $\delta$).
We recall that $\partial M$ is said to be \emph{outer-minimizing} in $(M,g)$
if its area is no greater than that of any enclosing surface;
here a closed surface $\Sigma \subset M$
\emph{encloses} $\partial M$
if $\partial M$ is contained
in the closure of the bounded component of $\R^3 \backslash \Sigma$.

Now given a $C^2$ metric $\gamma$ and a $C^1$ function $H$ on $S^2$,
we define the class
  \begin{equation}
  \label{PM}
    \mathcal{PM}[\gamma,H]
    :=
    \{ g \in \mathcal{PM} \; | \iota^*g = \gamma 
      \mbox{ and } \Hcal_{\iota}[g] \geq H \}
  \end{equation}
and the \emph{Bartnik-Bray outer mass} of $(S^2,\gamma,H)$
  \begin{equation}
  \label{mout}
    \mout[\gamma,H] 
    :=
    \inf \{ \madm[g] \; | \; g \in \mathcal{PM}[\gamma,H]
      \mbox{ and $\partial M$ is outer-minimizing in $(M,g)$} \},
  \end{equation}
where
  \begin{equation}
    \madm[g]
    :=
    \frac{1}{16\pi} \lim_{r \to \infty}
     \int_{\abs{x}=r} \left(g_{ik}^{\;\;\; ,k}-g^k_{\;\; k,i}\right)\frac{x^i}{\abs{x}}
  \end{equation}
is the ADM mass (\cite{ADM},\cite{BarADM},\cite{Chr},\cite{OM}) of $(M,g)$.
Here, as throughout the article except when otherwise indicated,
integrals over spheres in $\R^3$
are defined using the measure induced by the ambient Euclidean metric $\delta$.
Given a Riemannian $3$-manifold-with-boundary $(N,g_0)$
diffeomorphic to the closed unit ball $B$, we also define
the Bartnik-Bray outer mass of $(N,g_0)$
  \begin{equation}
    \mout[N,g_0]:=\mout[\iota^*\phi^*g_0, \Hcal_{\iota}[\phi^*g_0]]
  \end{equation}
with $\phi: B \to N$ any diffeomorphism.

As some motivation for the above definitions
we mention that the imposed equality on the induced metric and inequality on the mean curvature
ensure nonnegative distributional scalar curvature
of the extension of $g_0$ by $g$ as well as the applicability
of a version (\cite{MiaoCorners},\cite{ShiTam})
of the positive mass theorem to this extension; see \cite{MiaoVarEff}
for a thorough discussion.
We call the pair $(\gamma,H)$ \emph{Bartnik data}.

We can now state our results.
The definition of $\mout$, like that of the original Bartnik mass,
does not directly lend itself to straightforward evaluation.
At present its exact value is known only
on spherically symmetric Bartnik data (\cite{HI})
and on apparent horizons ($H=0$) satisfying a natural nondegeneracy condition (\cite{MS}).
Our main theorem gives a first-order estimate for Bartnik data
close to $(\iota^*\delta,-2)$.

\begin{theorem}
\label{mouthm}
There exist $\epsilon,C>0$ such that
the Bartnik-Bray outer mass of $(S^2,\gamma,H)$
(as defined by equation (\ref{mout}))
satisfies the estimate
  \begin{equation}
  \label{moutest}
    \abs{\mout[\gamma,H] - \int_{S^2} (6+2H-\gamma^\mu_{\;\; \mu})}
    \leq
    C\norm{\gamma-\iota^*\delta}_{C^4}^2+C\norm{H+2}_{C^3}^2
  \end{equation}
whenever $\norm{\gamma-\iota^*\delta}_{C^4}+\norm{H+2}_{C^3}<\epsilon$.
Here $\iota^*\delta$ is the standard metric on the unit sphere $S^2$,
which defines the trace of $\gamma$, the measure for the integral,
and the norms on the right-hand side.
\end{theorem}

In particular the estimate assumes that the data has almost constant-mean-curvature.
See \cite{CPCMM}
for a different estimate
that applies to the exactly constant-mean-curvature case
but makes no smallness assumption.
As an application of Theorem \ref{moutest}
we get the following asymptotic estimate for the outer mass of small metric spheres.

\begin{cor}
\label{sscor}
Let $p$ be a point in a smooth Riemannian manifold $(M,g)$.
Write $R$ for the scalar curvature of $g$
and $B_r$ for the closed geodesic ball of radius $r$ and center $p$.
Then
  \begin{equation}
  \label{firstlim}
    \lim_{r \to 0} \frac{\mout[B_r,g]}{r^3} = \frac{R(p)}{12}.
  \end{equation}
If moreover the Riemann curvature tensor of $g$ vanishes at $p$, then
  \begin{equation}
  \label{secondlim}
    \lim_{r \to 0} \frac{\mout[B_r,g]}{r^5} = \frac{\Delta R(p)}{120},
  \end{equation}
where $\Delta$ is the Laplacian defined by $g$.
\end{cor}

Of physical interest is the case when $(M,g)$ arises
as time-symmetric initial data for the Einstein equations,
where time symmetry means simply that $M$
is totally geodesic in the corresponding spacetime solution
and implies that the scalar curvature of $M$ is $16\pi$ times
the energy density contributed by all fields other than gravity itself.
The first limit of Corollary \ref{sscor} then states that the mass-to-volume
ratio of concentric spheres shrinking to a point tends to the energy density
at their center.
This result can be compared with corresponding ones
(such as in \cite{HorSch}, \cite{BrownLauYork}, \cite{Yu}, \cite{FanShiTam}, and \cite{ChenWangYau})
for other quasilocal masses
and is physically natural from the point of view of the equivalence principle,
which dictates that nongravitational sources of mass must dominate gravitational ones
on small neighborhoods of any given point in the slice.

Of course one would also like to probe higher-order contributions
to the outer mass of small spheres,
but the estimate (\ref{moutest}) can provide only limited information.
Corollary \ref{sscor} is proved in Section 3
by applying Theorem \ref{mouthm}
to the boundary of $B_r$ under the metric $r^{-2}g$,
Taylor expanding in $r$,
and afterward appropriately rescaling the mass.
The leading term of the integral in (\ref{moutest}) then yields (\ref{firstlim}).
The difficulty is that the next term (before rescaling) is of order $r^4$,
while the corresponding error
appearing on the right-hand side of (\ref{moutest})
is also $O\left(r^4\right)$, with coefficient bounded above by a constant times
the norm at $p$ of the Riemann curvature tensor associated to $g$.
Thus we deduce (\ref{secondlim}) only by assuming that $g$ is flat at $p$.
Even under this already undesirable restriction,
both (\ref{firstlim}) and (\ref{secondlim}) vanish
whenever $g$ is scalar-flat in a neighborhood of $p$,
so Corollary \ref{sscor} does not identify the leading (purely gravitational)
contributions to the outer mass in vacuum,
though this information is available for other quasilocal masses
(as in the references just cited).

Theorem \ref{mouthm} itself is proved by identifying a lower bound and an upper bound
for the outer mass that agree to first order in the data's deviation
from $(\iota^*\delta,-2)$. It is a well-known consequence of \cite{HI}
that the Hawking mass provides a lower bound.
For the upper bound,
encouraged by Bartnik's static extension conjecture,
we use the mass of a static vacuum extension.
Since this conjecture merely motivates our proof
but plays no role in its details,
we will not review the conjecture
and simply refer the reader to
\cite{BarNew},\cite{BarTHL}, and \cite{BarICM}
for the original statements of the conjecture
and to \cite{CorDef}, \cite{MiaoVarEff}, \cite{MS}, and \cite{AJ}
for some interesting results on its status and for further references.

We recall that a \emph{static spacetime}
(Chapter XIV of \cite{CB}) is
(by one of several slightly differing definitions) a Lorentzian manifold
possessing a complete, timelike, hypersurface-orthogonal Killing field 
and that a Riemannian metric $g$ on a manifold $M$ is called \emph{static}
(\cite{CB}, and Section 2 of \cite{CorDef})
if $(M,g)$ is a slice orthogonal to such a Killing field 
in some static spacetime satisfying the Einstein equations,
possibly with matter.
Equivalently, there is a positive function $\Phi$ on $M$,
representing the absolute value of the Lorentzian length of the Killing field,
such that the metric $-\Phi^2 \, dt^2 + g$ on $\R \times M$
is, in the vacuum case, Ricci flat.
The evolution and constraint equations in turn
imply the intrinsic characterization that $(M,g)$ is static vacuum if and only if
there is a function $\Phi$ so that together
$g$ and $\Phi$ solve the \emph{static vacuum equations}
  \begin{equation}
  \label{statvacsys}
    \begin{aligned}
      D_g^2\Phi - \Phi \Ric[g] = 0 \\
      \Delta_g \Phi = 0,
    \end{aligned}
  \end{equation}
where $D_g^2\Phi$, $\Delta_g \Phi$, and $\Ric[g]$ are respectively
the $g$ Hessian of $\Phi$, the $g$ Laplacian of $\Phi$, and the Ricci curvature of $g$.
In this context the lapse function $\Phi$ is called a \emph{static potential}
(because for weak fields it approximates the Newtonian gravitational potential)
for $(M,g)$.

We will call the boundary value problem given by the system (\ref{statvacsys})
with prescribed Bartnik boundary data the \emph{static extension problem}.
The existence of a solution to the static extension problem
in our small-data regime is not a new result.
The problem was solved first by Miao in \cite{MiaoStatExt},
with a symmetry assumption on the boundary data,
and later in generality by Anderson
(whose proof can be extended to higher dimensions)
in \cite{AndLoc},
applying the framework developed in \cite{AndBdy} and \cite{AndKhu}.
Anderson's existence proof simultaneously establishes uniqueness
(up to diffeomorphism)
within the class of static metrics $(g,\Phi)$ close to $(\delta,1)$.

In the next section we present a third construction of the extension,
which is sufficiently explicit to permit the calculation
of the extension's mass to first order in its boundary data's deviation
from a standard sphere in Euclidean space. 
In the final section we apply this estimate in the proof of Theorem \ref{moutest}
and we end with the proof of Corollary \ref{sscor}.

\subsection*{Acknowledgments}
I thank Rick Schoen for introducing me to Bartnik mass
and for very kindly sharing his invaluable perspective
on certain aspects of the problem treated here.
I also benefited greatly from a number of clarifying conversations
on related topics
with Hung Tran, Thomas Murphy, Alex Mramor, Peter McGrath, and especially Nicos Kapouleas.

\section{The static vacuum extension and its mass}

Let $M$ denote the complement in $\R^3$ of the open unit ball centered at the origin,
equipped with its natural structure as a manifold with boundary,
and let $\iota: S^2 \to M$ be the inclusion map for the boundary $\partial M=S^2$,
the origin-centered unit sphere.
Writing $\Met^2_{loc}(\cdot)$ for the space of $C^2_{loc}$ Riemannian metrics
on a manifold with (possibly empty) boundary,
we define the operators
$\Scal: \Met^2_{loc}(M) \times C^2_{loc}(M) 
 \to C^0_{loc}\left(T^*M^{\odot 2}\right) \times C^0_{loc}(M)$
and $\Bcal: \Met^2_{loc}(M) \to \Met^2_{loc}(\partial M) \times C^1_{loc}(\partial M)$
by
  \begin{equation}
    \Scal \begin{bmatrix} g \\ \Phi \end{bmatrix}
    =
    \begin{bmatrix} D_g^2 \Phi - \Phi \Ric[g] \\ \Delta_g \Phi \end{bmatrix}
    \mbox{ and }
    \Bcal[g]
    =
    \begin{bmatrix} \iota^*g \\ \Hcal_{\iota}[g] \end{bmatrix},
  \end{equation}
where $\Ric[g]$, $D_g$, and $\Delta_g$
respectively denote the Ricci curvature, Levi-Civita connection, and Laplacian of $g$,
and where $\Hcal_{\iota}[g]$ denotes the scalar mean curvature of $\partial M$ relative to $g$ 
and the corresponding unit normal directed toward the interior of $M$
(recalling that by our sign convention
the $g$ divergence of this unit normal is $-\Hcal_{\iota}[g]$).
Then $g$ and $\Phi$ solve the static vacuum extension problem on $M$ with boundary data
$\gamma \in \Met^2_{loc}(\partial M)$ and $H \in C^1_{loc}(\partial M)$
if and only if
  \begin{equation}
  \label{statext}
    \Scal \begin{bmatrix} g \\ \Phi \end{bmatrix} = \begin{bmatrix} 0 \\ 0 \end{bmatrix}
    \mbox{ and }
    \Bcal[g] = \begin{bmatrix} \gamma \\ H \end{bmatrix}.
  \end{equation}

Of course the Euclidean metric $\delta$ on $M$ is static, with potential $\Phi=1$,
so for boundary metric $\gamma$ close to the round metric $\iota^*\delta$ and boundary
mean curvature $H$ close to $-2$, we seek an extension of the form
$g=\delta + \eta$ with $\Phi=1+u$, 
where $\eta$ is a small section of $T^*M^{\odot 2}$ decaying rapidly
enough for $g$ to be asymptotically flat
and $u$ is similarly a small function with appropriate decay.

We now define some standard norms and corresponding Banach spaces
of sections of various tensor bundles.
Let $N$ be a $3$-dimensional submanifold (possibly with boundary) of $\R^3$;
as examples of $N$ we have in mind primarily $M$ and $\R^3$ itself.
Given a section $F$ of a tensor bundle over $N$,
a nonnegative integer $j$,
and real numbers $\alpha \in [0,1)$ and $\beta \in (0,\infty)$,
we will make use of the
H\"{o}lder norm
  \begin{equation}
    \norm{F}_{j,\alpha}
    :=
    \sup_{\vec{x} \neq \vec{y} \in N}
      \frac{\abs{D_\delta^jF(\vec{x})- D_\delta^jF(\vec{y})}_\delta}{\abs{\vec{x}-\vec{y}}^\alpha}
      + \sum_{i=0}^j \sup_{\vec{x} \in N} \abs{D_\delta^iF(\vec{x})}_{\delta},
  \end{equation}
as well as the weighted H\"{o}lder norm
  \begin{equation}
    \norm{F}_{j,\alpha,\beta} 
    :=
    \sup_{\vec{x} \neq \vec{y} \in N} 
      \left[1+\min \left\{\abs{\vec{x}},\abs{\vec{y}}\right\}\right]^{\alpha+\beta+j}
        \frac{\abs{D^jF(\vec{x})- D^jF(\vec{y})}_\delta}
        {\abs{\vec{x}-\vec{y}}^\alpha}
      + \sum_{i=0}^j \sup_{\vec{x} \in N} 
        (1+\abs{\vec{x}})^{\beta+i}\abs{D_\delta^iF(\vec{x})}_{\delta},
  \end{equation}
where the derivatives $D_\delta$ and differences are taken componentwise relative
to the standard Cartesian coordinates $\{x^1,x^2,x^3\}$ on $\R^3$.
Given instead a section $F$ of a tensor bundle over $\partial M=S^2$ we define the
standard H\"{o}lder norm
  \begin{equation}
    \norm{F}_{j,\alpha}
    :=
    \sup_{\vec{x}\neq \pm \vec{y} \in \partial M}
      \frac{\abs{D_{\iota^*\delta}^jF(\vec{x})
          -P_{\vec{y}}^{\vec{x}}D_{\iota^*\delta}^jF(\vec{y})}_{\iota^*\delta}}
        {\abs{\vec{x}-\vec{y}}^\alpha}
      + \sum_{i=0}^j \sup_{\vec{x} \in N} \abs{D_{\iota^*\delta}^iF(\vec{x})}_{\iota^*\delta},
  \end{equation}
where $D_{\iota^*\delta}$ denotes covariant differentiation induced by 
the round metric $\iota^*\delta$
and where $P_{\vec{y}}^{\vec{x}}$ denotes parallel transport,
likewise induced by $\iota^*\delta$,
from the fiber over $\vec{y}$ to the fiber over $\vec{x}$
along the unique $\iota^*\delta$ minimizing geodesic joining the two points.
For each tensor bundle $E$ over $N$ or $\partial M$ we also define the Banach spaces
$C^{j,\alpha,\beta}(E)$ and $C^{j,\alpha}(E)$
(written simply $C^{j,\alpha,\beta}(N)$ and $C^{j,\alpha}(N)$ as usual
when $E$ is the trivial bundle $N \times \R$ (and likewise for $\partial N \times \R$))
of sections of $E$ with finite $\norm{\cdot}_{j,\alpha,\beta}$ 
and $\norm{\cdot}_{j,\alpha}$ norms respectively.

Linearizing the above operators $\Scal$ and $\Bcal$ about $g=\delta$ and $\Phi=1$ we define
  \begin{equation}
  \label{Sdot}
  \begin{gathered}
  \Sdot: C^{j+2,\alpha,\beta}\left(T^*M^{\odot 2}\right) 
  \oplus C^{j+2,\alpha,\beta}(M)
  \to C^{j,\alpha,\beta+2}\left(T^*M^{\odot 2}\right) \oplus C^{j,\alpha,\beta+2}(M) \mbox{ by} \\
    \Sdot \begin{bmatrix} \eta \\ u \end{bmatrix}
    :=
    \left.\frac{d}{dt}\right|_{t=0} \Scal \begin{bmatrix} \delta + t\eta \\ 1 + tu \end{bmatrix}
    =
    \begin{bmatrix} D_\delta^2 u - \dot{\Ric}[\eta] \\ \Delta_\delta u \end{bmatrix}
  \end{gathered}
  \end{equation}
and
  \begin{equation}
  \begin{gathered}
  \Bdot: C^{j+2,\alpha,\beta}\left(T^*M^{\odot 2}\right)
  \to C^{j+2,\alpha}\left(T^*\partial M^{\odot 2}\right) \oplus C^{j+1,\alpha}(\partial M) \mbox{ by} \\
    \Bdot[\eta]
    :=
    \left.\frac{d}{dt}\right|_{t=0} \Bcal[\delta + t\eta]
    =
    \begin{bmatrix}
      \iota^*\eta \\
      \dot{\Hcal}_{\iota}[\eta]
    \end{bmatrix},
  \end{gathered}
  \end{equation}
where  $\dot{\Ric}[\eta]=\left.\frac{d}{dt}\right|_{t=0} \Ric[\delta + t\eta]$ 
and $\dot{\Hcal}_{\iota}[\eta]=\left.\frac{d}{dt}\right|_{t=0} \Hcal_{\iota}[\delta + t\eta]$ 
are the linearizations about $\delta$, with respect to the ambient metric,
of the Ricci curvature of $M$ and the mean curvature of $\partial M$ respectively.

We will solve the general linearized problem
modulo a Lie derivative of the Euclidean background metric
(arising from a well-known obstruction explained in the proof of Lemma \ref{inhomlinsol}).
We will then apply the linearized solution operator iteratively,
via the contraction mapping lemma,
to the original nonlinear problem.
Ultimately we will see that the Lie derivative term in the resulting fixed point
must vanish, on account of the way it was selected
together with constraints coming from the structure of the static vacuum equations.
The first application of the linearized solution operator
will be to the homogeneous problem with inhomogeneous boundary data.
Subsequent applications will feature the inhomogeneous problem
with data originating from error in the linear approximation;
since this error (estimated in the proof of Proposition \ref{statext})
is quadratic in the linear approximation itself,
we can in particular assume fast decay at infinity.
The following lemma addresses this inhomogeneous situation,
without attempting to control the boundary data.

\begin{lemma}
\label{inhomlinsol}
For any $\alpha, \beta \in (0,1)$ there exists a bounded linear map
  \begin{equation}
    \Ucal_{\Scal}: C^{1,\alpha,3+\beta}\left(T^*M^{\odot 2}\right)
      \oplus C^{1,\alpha,3+\beta}(M)
    \to
    C^{3,\alpha,1}\left(T^*M^{\odot 2}\right)
      \oplus C^{3,\alpha,1+\beta}(M)
      \oplus C^{2,\alpha,2+\beta}(TM)
  \end{equation}
such that if
  \begin{equation}
    \begin{bmatrix} \eta_{ab} \\ u \\ \chi^a \end{bmatrix}
    =
    \Ucal_{\Scal} \begin{bmatrix} S_{ab} \\ \sigma \end{bmatrix},
  \end{equation}
then
  \begin{equation}
  \label{ihlineqn}
    \Sdot \begin{bmatrix} \eta_{ab} \\ u \end{bmatrix}
    =
    \begin{bmatrix} 
      S_{ab} + \chi_{a;b} + \chi_{b;a} \\ \sigma
    \end{bmatrix},
  \end{equation}
where each semicolon indicates covariant differentiation relative to $\delta$,
and moreover there exists a constant $C>0$ depending on just $\alpha$ and $\beta$
such that
  \begin{equation}
  \label{chiest}
    \norm{\chi}_{2,\alpha,2+\beta} \leq C\norm{\Delta_\delta \chi}_{0,\alpha,4+\beta},
  \end{equation}
where $\Delta_\delta \chi^a=\chi^{a;c}_{\;\;\;\;\; ;c}$. 
\end{lemma}

\begin{proof}
Let $S \in C^{1,\alpha,3+\beta}\left(T^*M^{\odot 2}\right)$
and $\sigma \in C^{1,\alpha,3+\beta}(M)$.
Referring to (\ref{Sdot}), we see that
to solve (\ref{ihlineqn}) we need to pick a solution $u$
to the Poisson equation $\Delta_\delta u=\sigma$
and we need to find $\eta_{ab}$ and $\chi_a$ satisfying
  \begin{equation}
  \label{pric}
    \dot{\Ric}[\eta]_{ab}=u_{;ab}-S_{ab}-\chi_{a;b}-\chi_{b;a}.
  \end{equation}
In particular we have to confront the problem of prescribing linearized Ricci curvature.
Of course we would really like to solve the problem with $\chi=0$,
but there are well understood obstructions to the prescription of linearized Ricci
which force us to introduce $\chi$ in order to arrange for the right-hand side of
(\ref{pric}) to satisfy certain integrability conditions.

In fact, for a given symmetric $2$-tensor $T_{ab}$, the equation
$\dot{\Ric}[\eta]=T$ is equivalent to
  \begin{equation}
  \label{pein}
    \dot{G}[\eta]=\hat{T},
  \end{equation}
where $\hat{T}_{ab}=T_{ab}-\frac{1}{2}T^{c}_{\;\;c}\delta_{ab}$
and $\dot{G}[\eta]$ is the linearization of
the Einstein tensor $G_{ab}[g]=R_{ab}[g]-\frac{1}{2}R^c_{\;\;c}[g]g_{ab}$
at $\delta$
in the direction $\eta$. This same operator is known
(see \cite{Eastwood} for example) to arise in a different guise in elasticity theory;
one easily verifies that $\dot{G}$ may be written in the double-curl form
  \begin{equation}
  \label{dcurl}
    \dot{G}[\eta]_{ab}=\frac{1}{2}\epsilon_{acd}\eta^{de;cf}\epsilon_{feb},
  \end{equation}
where $\epsilon_{abc}$ is the Euclidean volume form
(whose components relative to right-handed Cartesian coordinates on $M$
are given by the totally antisymmetric Levi-Civita symbol $\epsilon_{ijk}$).

The linearization at $\delta$ of the twice contracted differential Bianchi identity
presents a necessary condition for the solvability of (\ref{pein}),
namely $\hat{T}$ must have vanishing divergence:
  \begin{equation}
  \label{transverse}
    \hat{T}_{c}^{\;\;;c}=0.
  \end{equation}
In $\R^3$ this condition is also sufficient, but on $M$ there are additional obstructions,
arising from its nontrivial topology.
In \cite{Gurtin}, working with the form (\ref{dcurl}),
it is proved that, on a bounded domain $\Omega$ of $\R^3$
with sufficiently smooth (but not necessarily connected) boundary,
the equation (\ref{pein}), for sufficiently smooth data $\hat{T}$,
admits a solution if and only if in addition to (\ref{transverse}) holding on $\Omega$
we also have
  \begin{equation}
  \label{equil}
    \int_{\partial \Omega} \hat{T}_{ab}n^b
    =
    \int_{\partial \Omega} \epsilon_{abc}x^b\hat{T}^{cd}n_d
    =0,
  \end{equation} 
where $n^a$ is a continuous unit normal on $\Omega$
and the integration measure is the natural one induced by $\delta$ on $\partial \Omega$.

A particularly simple proof of this same result appears in \cite{Carlson},
which is easily adapted to our situation to show that
for any $\hat{T} \in C^{1,\alpha,3+\beta}\left(T^*M^{\odot 2}\right)$
satisfying (\ref{transverse}) and (\ref{equil})
there exists $\eta \in C^{3,\alpha,1}\left(T^*M^{\odot 2}\right)$ solving (\ref{pein}).
In fact the proof of \cite{Carlson} goes through without any modification
once we know that
(i) every vector field in $C^{1,\alpha,3+\beta}(TM)$
with vanishing flux through $\partial M$
has a vector potential (another vector field with curl the given one)
in $C^{2,\alpha,2}(TM)$ 
and that in turn
(ii) every vector field in this last space,
again with vanishing flux through $\partial M$,
has a vector potential in $C^{3,\alpha,1}(TM)$.
These two facts are proved in \cite{NvW}.

Returning to our problem,
since the $\delta$ trace of (\ref{pric}) is
  \begin{equation}
    \dot{\Ric}[\eta]^c_{\;\;c}=\sigma-S^c_{\;\;c}-2\chi^c_{\;\;;c}
  \end{equation}
(recalling that $u$ will be picked to solve $\Delta_\delta u=\sigma$),
we find that (\ref{pric}) is equivalent to
  \begin{equation}
  \label{prein}
    \dot{G}[\eta]_{ab}=u_{;ab}-\frac{1}{2}\sigma \delta_{ab} - \hat{S}_{ab}
      - \chi_{a;b} - \chi_{b;a} + \chi^c_{\;\;;c}\delta_{ab},
  \end{equation}
where $\hat{S}_{ab}=S_{ab}-\frac{1}{2}S^c_{\;\;c}\delta_{ab}$.
In light of the preceding discussion
we will be able to solve this equation for $\eta$
provided we can choose $\chi$ and $u$ to make the right-hand side
satisfy the integrability conditions (\ref{transverse}) and (\ref{equil})
(with $\hat{T}$ replaced by the right-hand side of (\ref{prein})
and with $\partial \Omega$ replaced with $\partial M$).
Actually for
$\chi^a \in C^{2,\alpha,2+\beta}(TM)$ and $u \in C^{3,\alpha,1+\beta}(M)$,
the divergence-free condition (\ref{transverse})
implies (\ref{equil}),
because in this case the right-hand side of (\ref{prein}) belongs to
$C^{1,\alpha,3+\beta}\left(T^*M^{\odot 2}\right)$,
so the rapid decay means that on large spheres in $M$ centered at the origin
the corresponding integrals in (\ref{equil}) tend to $0$,
and then the divergence theorem in conjunction with (\ref{transverse})
ensures (\ref{equil}) holds exactly on $\partial M$.

Taking the divergence of (\ref{prein})
(and again using the fact that we will choose $u$ solving $\Delta_\delta u=\sigma$,
we see we must find $\chi^a \in C^{2,\alpha,2+\beta}(TM)$ satisfying
  \begin{equation}
    \chi_{a;b}^{\;\;\;\;;b}=\frac{1}{2}\sigma_{;a}-\hat{S}_{ab}^{\;\;\;;b}.
  \end{equation}
Working relative to Cartesian coordinates
this system simply consists of three independent
Poisson equations for the components of $\chi$.
Lemma \ref{Pfd} in Appendix \ref{Peqn} provides the existence
of a solution satisfying (\ref{chiest}).
Using the same lemma we also find $u$ solving $\Delta_\delta u=\sigma$
and satisfying
$\norm{u}_{3,\alpha,1+\beta} \leq C\norm{\sigma}_{1,\alpha,3+\beta}$.
Finally we can solve (\ref{prein}) for $\eta_{ab}$
as explained above,
using the proof of (\cite{Carlson})
supplemented with the results on vector potentials from (\cite{NvW});
the last reference also guarantees the estimates asserted for $\eta$
in terms of the data.
\end{proof}

The next lemma, whose second item will ultimately furnish our mass estimate,
incorporates the boundary data.
The heart of the proof is the recognition that, at the linear level,
the static condition can be maintained while making prescribed
perturbations to the boundary data through harmonic conformal transformations
and diffeomorphisms from $M$ to subsets of $\R^3$.

\begin{lemma}
\label{linsol}
Given $\alpha, \beta \in (0,1)$, there exists a bounded linear map
    \begin{equation}
    \begin{aligned}
    \Ucal:
        &C^{1,\alpha,3+\beta}\left(T^*M^{\odot 2}\right) 
        \oplus C^{1,\alpha,3+\beta}(M) 
        \oplus C^{3,\alpha}\left(T^*\partial M^{\odot 2}\right) 
        \oplus C^{2,\alpha}(\partial M) \\
    &\to
        C^{3,\alpha,1}\left(T^*M^{\odot 2}\right)
        \oplus C^{3,\alpha,1}(M)
        \oplus C^{2,\alpha,2+\beta}(TM)
    \end{aligned}
  \end{equation}
such that if
  \begin{equation}
    \begin{bmatrix} \eta_{ab} \\ u \\ \chi^a \end{bmatrix}
    =
    \Ucal \begin{bmatrix} S_{ab} \\ \sigma \\ \omega_{\mu \nu} \\ \kappa \end{bmatrix},
  \end{equation}
then (i)
  \begin{equation}
  \label{lineqn}
    \Sdot \begin{bmatrix} \eta_{ab} \\ u \end{bmatrix}
    =
    \begin{bmatrix} 
      S_{ab} + \chi_{a;b} + \chi_{b;a} \\ \sigma
    \end{bmatrix},
    \quad
    \Bdot[\eta_{ab}]
    =
    \begin{bmatrix} \omega_{\mu \nu} \\ \kappa \end{bmatrix},
    \quad \mbox{and }
    \norm{\chi}_{2,\alpha,2+\beta} \leq C\norm{\Delta_\delta \chi}_{0,\alpha,4+\beta},
  \end{equation}
where $C$ is a constant depending on only $\alpha$ and $\beta$,
and

(ii) in the special case that $S_{ab}=0$ and $\sigma=0$,
we have $\eta_{ab}=-2u \delta_{ab}$ outside a compact set
and moreover
  \begin{equation}
  \label{uavg}
    \int_{\partial M} u
    =
    \frac{1}{4} \int_{\partial M} \omega_{\mu}^{\;\; \mu} - \frac{1}{2} \int_{\partial M} \kappa, 
  \end{equation}
where $\omega$ is contracted via $\iota^*\delta$ and
the integration measure is given by
the standard area form on $S^2=\partial M$ (for a total area of $4\pi$),
induced by $\iota^*\delta$.
\end{lemma}

\begin{proof}

\textbf{(i)} Let $S \in C^{1,\alpha,3+\beta}\left(T^*M^{\odot 2}\right)$,
$\sigma \in C^{1,\alpha,3+\beta}(M)$,
$\omega \in C^{3,\alpha}\left(T^*\partial M^{\odot 2}\right)$,
and $\kappa \in C^{2,\alpha}(\partial M)$.
Appealing to Lemma \ref{inhomlinsol}
we define $\etilde_{ab}$, $w$, and $\chi_a$ by
  \begin{equation}
  \label{ewchdef}
    \begin{bmatrix} \etilde_{ab} \\ w \\ \chi^a \end{bmatrix}
    =
    \Ucal_{\Scal}\begin{bmatrix} S_{ab} \\ \sigma \end{bmatrix},
  \end{equation}
thereby securing
  \begin{equation}
  \label{sofar}
    \Sdot \begin{bmatrix} \etilde \\ w \end{bmatrix}
    =
    \begin{bmatrix} S_{ab} + \chi_{a;b} + \chi_{b;a} \\ \sigma \end{bmatrix},
  \end{equation}
with $\norm{w}_{3,\alpha,1}$ and $\norm{\etilde}_{3,\alpha,1}$
controlled by the data,
but it remains to enforce the boundary conditions.

To this end we will further alter the metric by infinitesimal diffeomorphism
of the form $\xi_{a;b}+\xi_{b;a}$
with $\xi^a$ a compactly supported vector field,
which of course does not affect the linearized Ricci tensor,
so preserves equation (\ref{sofar}),
but can be exploited to adjust the boundary data.
Fields tangential to $\partial M$ represent pure gauge changes,
which are nevertheless useful for reducing to the case of 
boundary metric conformally round to first order (infinitesimal uniformization),
while normal fields genuinely deform the boundary geometry.

In fact we restrict attention to vector fields of the form
  \begin{equation}
    \xi=\psi \xi^\perp \partial_r + \psi \xi^\top,
  \end{equation}
where $\xi^\perp$ is a (scalar-valued) function on $\partial M$,
$\xi^\top$ is a radially constant
($\delta$-parallel along rays through the origin)
vector field tangential to $\partial M$, 
and $\psi$ is a bump function
identically $1$ on $\{\abs{x} \leq 2 \}$
and identically $0$ on $\{\abs{x} \geq 3\}$.
Thus we consider variations of $\partial M$
arising as reparametrizations generated by $\xi^\top$
and as graphs generated by $\xi^\perp$.
To wit, writing $L$ for the Lie derivative, we have
  \begin{equation}
    \label{Bvec}
    \Bdot[\xi_{a;b}+\xi_{b;a}]
    =
    \begin{bmatrix}
      L_{\xi^\top}\iota^*\delta + 2\xi^\perp \iota^*\delta \\
      \left(\Delta_{\iota^*\delta}+2\right) \xi^\perp
    \end{bmatrix},
  \end{equation}
as can be verified either by direct calculation of the variation of the induced metric
and mean curvature under variations of the ambient metric
(included in Appendix \ref{varapp} for reference)
or by using the possibly more familiar formulas
for the variation of these quantities under variations of the immersion,
after applying the tautologies
$\iota^* \phi_t^*\delta = (\phi_t \circ \iota)^*\delta$
and $\Hcal_\iota[\phi_t^*\delta] = \Hcal_{\phi_t \circ \iota}[\delta]$
for the diffeomorphisms $\phi_t: M \to \R^3$ generated by $\xi$.

Of course variations by vector fields alone will not span the
tangent space of boundary data at $\delta$, and so
we will additionally avail ourselves of linearized conformal transformations generated
by harmonic (relative to the Euclidean metric $\delta$) functions vanishing at infinity.
For $v$ satisfying $\Delta_\delta v=0$ we get
  \begin{equation}
    \dot{\Ric}[v\delta]=-\frac{1}{2}D_\delta^2 v,
  \end{equation}
so we may simultaneously replace $\etilde$ and $w$ in (\ref{sofar})
by $\etilde + v\delta$ and $w-\frac{1}{2}v$ respectively,
again without altering the right-hand side.
As for the boundary data, the change of the induced metric is obvious
and the change of mean curvature is easy to calculate
(as shown in Appendix \ref{varapp}):
  \begin{equation}
    \label{Bconf}
    \Bdot[v\delta]=\begin{bmatrix} v\iota^*\delta \\ v - v_{,r} \end{bmatrix}.
  \end{equation}

We will soon see that
these two types of modifications will suffice to prescribe the boundary data,
and accordingly we seek $\xi^a$ and $v$ as just described
so that
  \begin{equation}
  \label{bdycon}
    \Bdot[\etilde_{ab} + \xi_{a;b}+\xi_{b;a} + v\delta_{ab}]
    =
    \begin{bmatrix} \omega_{\mu \nu} \\ \kappa \end{bmatrix}.
  \end{equation}

We write
$\Bdot_1: C^{j,\alpha,\beta}\left(T^*M^{\odot 2}\right)
  \to C^{j,\alpha}\left(T^*\partial M^{\odot 2}\right)$
and 
$\Bdot_2: C^{j+1,\alpha,\beta}\left(T^*M^{\odot 2}\right)
  \to C^{j,\alpha}(\partial M)$
for the first (metric) and second (mean curvature) components respectively of $\Bdot$.
Since (see for example Lectures 1 and 3 of \cite{Via})
every symmetric $2$-tensor on $S^2$ may be written as the sum
of a Lie derivative of $\iota^*\delta$
and a function times $\iota^*\delta$,
there exist $h \in C^{3,\alpha}(\partial M)$ and $W \in C^{4,\alpha}(T\partial M)$
such that
  \begin{equation}
  \label{hWdef}
    \omega - \Bdot_1[\etilde]  = h \iota^*\delta + L_{_W} \iota^*\delta.
  \end{equation}
We also define $k \in C^{2,\alpha}(\partial M)$ by
  \begin{equation}
  \label{kdef}
    \kappa-\Bdot_2[\etilde] = k.
  \end{equation}
Setting $\xi^\top=W$ and referring to equations (\ref{Bvec}) and (\ref{Bconf}), equation (\ref{bdycon}) becomes the system
  \begin{equation}
    \label{scalarsys}
    \begin{aligned}
      &\iota^*v+2\xi^\perp=h \\
      &\iota^*(v-v_{,r}) + \left(\Delta_{\iota^*\delta}+2\right)\xi^\perp=k.
    \end{aligned}
  \end{equation}

For each nonnegative integer $\ell$ and for each integer $m \in [-\ell,\ell]$
we let $\Y: S^2 \to \C$ be a spherical harmonic satisfying
$\Delta_{\iota^*\delta} \Y(\theta,\phi) = -\ell(\ell+1)\Y(\theta,\phi)$
and $-i \partial_\theta \Y(\theta,\phi) = mY(\theta,\phi)$ (where $\theta$ is the azimuth angle)
and chosen so that 
$\bigcup_ {\ell=0}^\infty \left\{ \Y \right\}_{m=-\ell}^\ell$
is an orthonormal basis for $L^2 \left(S^2,\iota^*\delta\right)$.

Introducing the coefficients defined by the expansions 
\begin{equation}
\label{sphdecomp}
  \begin{aligned}
    &h(\theta,\phi)=\sum_{\ell=0}^\infty \sum_{m=-\ell}^\ell \shc{h} \Y(\theta,\phi), \\
    &k(\theta,\phi)=\sum_{\ell,m} \shc{k} \Y(\theta,\phi), \\
    &v(r,\theta,\phi)=\sum_{\ell,m} \shc{v} r^{-\ell-1}\Y(\theta,\phi), \mbox{ and} \\
    &\xi^\perp(r,\theta,\phi)=\sum_{\ell,m} \shc{\xi} \Y(\theta,\phi),
  \end{aligned}
  \end{equation}
we find that $v$ and $\xi^\perp$ solve our system (\ref{scalarsys}) if and only if
  \begin{equation}
  \label{bdymodes}
    \begin{bmatrix} 1 & 2 \\ \ell + 2 & -(\ell-1)(\ell+2) \end{bmatrix}
      \begin{bmatrix} \shc{v} \\ \shc{\xi} \end{bmatrix}
    =
    \begin{bmatrix} \shc{h} \\ \shc{k} \end{bmatrix}    
  \end{equation}
for each nonnegative integer $\ell$ and for each integer $m \in [-\ell,\ell]$.
Evidently for each such $\ell$ the matrix in equation (\ref{bdymodes}) is invertible,
so we get the unique solution
  \begin{equation}
  \label{fourier}
    \begin{bmatrix} \shc{v} \\ \shc{\xi} \end{bmatrix}
    =
    \begin{bmatrix}
      \frac{\ell-1}{\ell+1} & \frac{2}{(\ell+1)(\ell+2)} \\
      \frac{1}{\ell+1} & \frac{-1}{(\ell+1)(\ell+2)}
    \end{bmatrix}
    \begin{bmatrix} \shc{h} \\ \shc{k} \end{bmatrix}.
  \end{equation}
Through (\ref{sphdecomp})
these coefficients define distributional solutions
$v$, $\xi^\perp$
to (\ref{scalarsys}).
H\"{o}lder estimates can be obtained
from integral representation of $v$ and $\xi^\perp$
as follows.

First, given $f \in C^{2,\alpha}(S^2,\iota^*\delta)$,
we write $P[f]$ for the unique harmonic function on the
closed unit ball (the closure of $\R^3 \backslash M$)
agreeing with $f$ on $\partial M$.
For each $r \in [0,1)$ we also define the function
$P_r[f]$ on $S^2$ by $P_r[f](\vec{u})=P[f](r\vec{u})$.
In particular
  \begin{equation}
  \label{Prbasis}
  P_r\left[Y_{\ell,m}\right]=r^\ell Y_{\ell,m}.
  \end{equation}
Then from (\ref{fourier}) we deduce that
  \begin{equation}
  \label{coeffPr}
  \begin{aligned}
    \shc{v}
    &=
    \frac{
        \left\langle
        \Y, \, 2k-(\Delta_{\iota^*\delta}+2)h
        \right\rangle}
      {(\ell+1)(\ell+2)} \\
    &=
    -\left\langle
    \int_0^1 \int_0^s P_r\left[\left(\Delta_{\iota*\delta}+2\right)\Y\right] \, dr \, ds,
      \, h
    \right\rangle
    +
    \left\langle
      \int_0^1 \int_0^s P_r\left[Y_{\ell,m}\right] \, dr \, ds,
       \, 2k
      \right\rangle \mbox{ and}  \\
    \shc{\xi}
    &=
    \frac{\left\langle \Y, \, h \right\rangle}{\ell+1}
    - \frac{\left\langle \Y, \, k \right\rangle}{(\ell+1)(\ell+2)}
    =
    \left\langle \int_0^1 P_r[\Y] \, dr , \, h \right\rangle
      - \left\langle \int_0^1 \int_0^s P_r[\Y] \, dr \, ds, \, k \right\rangle,
  \end{aligned}
  \end{equation}
where the angled brackets denote the $L^2\left(S^2,\iota^*\delta\right)$
inner product (which we take to be conjugate-linear in the left-hand factor).

From (\ref{Prbasis}) we see that $P_r$
extends to a self-adjoint map on $L^2(S^2)$,
so in turn the maps
$f \mapsto \int_0^1 P_r[f] \, dr$
and $f \mapsto \int_0^1 \int_0^s P_r[f] \, dr \, ds$
are also self-adjoint.
Additionally we have
$\left\langle \Delta_{\iota*\delta} \Y, f \right \rangle=
\left\langle \Y, \Delta_{\iota^*\delta}f \right\rangle$
whenever $f \in C^2(S^2)$,
and note that from standard elliptic regularity theory
we have
  \begin{equation}
  \label{stdreg}
    \norm{P[f]}_{j,\alpha} \leq C(j,\alpha) \norm{f}_{j,\alpha}.
  \end{equation}
Since $h \in C^{3,\alpha}(S^2)$ and $k \in C^{2,\alpha}(S^2)$,
we conclude using (\ref{coeffPr}) and (\ref{sphdecomp}),
that $v$ and $\xi^\perp$
are both at least $C^{1,\alpha}$ functions representable as
  \begin{equation}
  \label{vxirep}
  \begin{aligned}
    &v=P^{ext}\left[\int_0^1 \int_0^s
      P_r[2k]-(\Delta_{\iota^*\delta}+2)P_r[h] \, dr \, ds\right]
    \mbox{ and} \\
    &\xi^\perp=\int_0^1 P_r[h] \, dr - \int_0^1 \int_0^s P_r[k] \, dr \, ds,
  \end{aligned}
  \end{equation}
where $P^{ext}[f]$ is the exterior harmonic extension of $f \in C^0(\partial M)$,
that is the unique bounded function on $M$
harmonic on the interior and agreeing with $f$ on the boundary.

To establish the higher regularity with estimates as asserted
in the statement of the lemma
we need better bounds for $P_r[f]$
than (\ref{stdreg}),
which we will obtain by a slight variant of a standard derivation
of interior gradient estimates for harmonic functions,
as follows.
Let $\vec{x}$ belong to the unit ball in $\R^3$ centered at $0$
and suppose $f \in C^{2,\alpha}(S^2)$.
Since $P[f]$ is harmonic, each of its coordinate partial derivatives
$\partial_iP[f]$ is too, so by the mean value property
and the divergence theorem
  \begin{equation}
  \label{C1est}
    \begin{aligned}
      \abs{\partial_iP[f](\vec{x})}
      &=
      \abs{\frac{6}{\pi(1-\abs{\vec{x}})^3}
        \int_{\abs{\vec{y}-\vec{x}} \leq \frac{1}{2}(1-\abs{\vec{x}})}
          \partial_{y^i}\left(P[f](\vec{y})-P[f]\left(\frac{1+\abs{\vec{x}}}{2}\vec{x}\right)\right)
          \, d^3y} \\
      &\leq
      2^{-\alpha/2} \cdot 12\norm{P[f]}_{C^{0,\alpha}
       (S\left[\vec{x},\frac{1}{2}(1-\abs{\vec{x}})\right])}(1-\abs{\vec{x}})^{\alpha-1}
        \int_0^\pi (1-\cos \phi)^{\alpha/2} \sin \phi \, d\phi \\
      &=
      \frac{48}{2+\alpha}\norm{P[f]}_{C^{0,\alpha}
        (S\left[\vec{x},\frac{1}{2}(1-\abs{\vec{x}})\right])}
        (1-\abs{\vec{x}})^{\alpha-1},
    \end{aligned}
  \end{equation}
where $S[\vec{v},r]$ is the sphere with center $\vec{v}$ and radius $r$.

Next, for any two points $\vec{y},\vec{z} \in S\left[\vec{x},\frac{1}{2}(1-\abs{\vec{x}})\right]$,
again using the mean value property and the divergence theorem,
  \begin{equation}
  \label{C1alphaest}
    \begin{aligned}
      \abs{\partial_iP[f](\vec{y})-\partial_i[f](\vec{z})}
      &=
      \abs{\frac{6}{\pi(1-\abs{\vec{x}})^3}
        \int_{\abs{\vec{w}-\vec{y}} \leq \frac{1}{2}(1-\abs{\vec{x}})}
          \partial_{w^i}\left(P[f](\vec{w})-P[f](\vec{w}+\vec{z}-\vec{y})\right) \, d^3w} \\
      &\leq
      24(1-\abs{\vec{x}})^{-1}\norm{P[f]}_{C^{0,\alpha}(S[\vec{x},1-\abs{\vec{x}}])}
        \abs{\vec{w}-\vec{z}}^\alpha.
    \end{aligned}
  \end{equation}
Now a simple inductive argument using
(\ref{C1est}), (\ref{C1alphaest}), and (\ref{stdreg}) yields
  \begin{equation}
    \norm{P_r[f]}_{j+k,\alpha} \leq C(j,k,\alpha)(1-r)^{\alpha-k}\norm{f}_{j,\alpha},
  \end{equation}
whence it follows immediately that
  \begin{equation}
    \norm{\int_0^1 P_r[f] \, dr}_{j+1,\alpha} +
    \norm{\int_0^1 \int_0^s P_r[f] \, dr \, ds}_{j+2,\alpha} \leq C(\alpha,j)\norm{f}_{j,\alpha},
  \end{equation}
which applied to (\ref{vxirep}) in conjunction with the classical estimate
$\norm{P^{ext}[f]}_{j,\alpha,1} \leq C(j)\norm{f}_{j,\alpha}$
ensures that for some $C>0$
  \begin{equation}
    \norm{v}_{3,\alpha,1}+\norm{\xi^\perp}_{4,\alpha}
    \leq
    C\left(\norm{h}_{3,\alpha}+\norm{k}_{2,\alpha}\right)
  \end{equation}
as needed.

The proof of (i) is now complete, with the solution operator defined by
  \begin{equation}
    \label{Udef}
    \Ucal \begin{bmatrix} S_{ab} \\ \sigma \\ \omega_{\mu \nu} \\ \kappa \end{bmatrix}
    =
    \begin{bmatrix} \etilde_{ab}+\xi_{a;b}+\xi_{b;a}+v\delta_{ab} \\ w-\frac{1}{2}v \\ \chi^a \end{bmatrix}.
  \end{equation}

\textbf{(ii)} We first observe that when $S_{ab}$ and $\sigma$ vanish identically,
so do $\chi^a$, $w$, and $\etilde$
(by virtue of (\ref{ewchdef})
and the boundedness of $\Ucal_{\Scal}$),
so that equations (\ref{hWdef}) and (\ref{kdef}) read simply
$\omega_{\mu \nu}=h\left(\iota^*\delta\right)_{\mu \nu}+W_{\mu:\nu}+W_{\nu:\mu}$ and $\kappa=k$,
with each colon indicating covariant differentiation according to $\iota^*\delta$.
Thus in particular $\omega_{\mu}^{\;\; \mu}=2h+2W_{\mu}^{\;\; :\mu}$,
so $\int_{S^2} \omega_{\mu}^{\;\; \mu}=2\int_{S^2} h$
and of course $\int_{S^2} \kappa = \int_{S^2} k$.
Referring to the linearized potential $u=0-\frac{1}{2}v$ appearing in (\ref{Udef})
and to the top row of (\ref{fourier}) at $\ell=m=0$, we finish with
  \begin{equation}
    \int_{S^2} u = -\frac{1}{2} \int_{S^2} v = \frac{1}{2} \int_{S^2} h - \frac{1}{2} \int_{S^2} k
      =
      \frac{1}{4}\int_{S^2} \omega_{\mu}^{\;\; \mu} - \frac{1}{2} \int_{S^2} \kappa
  \end{equation}
as claimed.

\end{proof}

Taking $S=0$, $\sigma=0$, $\omega=\gamma-\iota^*\delta$, and $\kappa=H+2$ in the lemma,
from $\Ucal$ we obtain $\eta_{ab}$ and $u$ (and $\chi=0$) 
solving our extension problem to first order:
  \begin{equation}
    \label{firstordersol}
    \Sdot \begin{bmatrix} \eta \\ u \end{bmatrix} = \begin{bmatrix} 0 \\ 0 \end{bmatrix}
    \mbox{ and }
    \Bdot[\eta] = \begin{bmatrix} \gamma-\iota^* \delta \\ H+2 \end{bmatrix}.
  \end{equation}
The proof of the next
proposition provides the higher-order corrections by using the contraction mapping lemma.
Note that this first-order application
of (\ref{linsol}) requires no $\chi^a$ to adjust the source terms,
but subsequent applications, aimed to cancel the nonlinear terms,
a priori may, spoiling the static condition.
To the contrary, 
an argument along the lines of Miao's proof of
his Reduction Lemma in \cite{MiaoStatExt} 
shows that for sufficiently small data $\chi^a$ vanishes identically,
leaving an exact solution.

As mentioned earlier, the existence assertion of the proposition below
was first proved by Miao in \cite{MiaoStatExt}, with a symmetry condition on the data,
while the general existence and uniqueness statements were established by Anderson
in \cite{AndLoc};
the analyticity, relative to harmonic local coordinates, of static vacuum metrics
was proved by M\"{u}ller zum Hagen in \cite{MzH}.
The novelty of the proposition is its estimate of the extension's ADM mass,
which is made by referring to item (ii) of the lemma
to compute the mass of the linearized solution
and by using the estimate for the nonlinear corrections obtained
in the course of constructing the exact solution.

\begin{prop}
\label{statext}
Let $\alpha \in (0,1)$.
There exist $C,\epsilon>0$ such that whenever
$\norm{\gamma_{\mu \nu}-(\iota^*\delta)_{\mu \nu}}_{3,\alpha}+\norm{H+2}_{2,\alpha}<\epsilon$,
(i) there is a $C^{3,\alpha}$ asymptotically flat Riemannian metric $g$ on $M$
and there is a $C^{3,\alpha}$ function $\Phi$ on $M$ solving
the static vacuum extension problem
  \begin{equation}
  \label{statvacext}
    \Scal \begin{bmatrix} g \\ \Phi \end{bmatrix}
    =
    \begin{bmatrix} 0 \\ 0 \end{bmatrix}
    \mbox{ and }
    \Bcal[g]
    =
    \begin{bmatrix} \gamma_{\mu \nu} \\ H \end{bmatrix},
  \end{equation}
with
  \begin{equation}
    \norm{g-\delta}_{3,\alpha,1}+\norm{\Phi-1}_{3,\alpha,1}
    \leq
    C\norm{\gamma-\iota^*\delta}_{3,\alpha}+C\norm{H+2}_{2,\alpha};
  \end{equation}
(ii) there is a neighborhood of $(\delta,1)$ in 
$C^{3,\alpha,\beta}(T^*M^{\odot 2}) \times C^{3,\alpha,\beta}(M)$
within which any two solutions of (\ref{statvacext}) are diffeomorphic;
(iii) there are global harmonic coordinates for the interior of $(M,g)$
with respect to which $g$ and $\Phi$ are analytic;
(iv) $(g,\Phi)$ admits no closed minimal surfaces
and $\partial M$ is outer-minimizing;
and (v) the ADM mass $\madm[g]$ of the extension $g$ satisfies the estimate
  \begin{equation}
    \madm[g] = \frac{1}{16\pi} \int_{\partial M} \left( 6 + 2H - \gamma_{\mu}^{\;\; \mu} \right)
      + O(\norm{\gamma-\iota^*\delta}_{3,\alpha}^2)+O(\norm{H+2}_{2,\alpha}^2).
  \end{equation}
\end{prop}

\begin{proof}
\textbf{(i)} 
We fix some $\beta \in (0,1)$ (so we can apply Lemma \ref{linsol}).
We seek a solution of the form
$g_{ab}=\delta_{ab}+\eta_{ab}+\theta_{ab}$ and $\Phi=1+u+v$, 
with $\eta$ and $u$ obtained from Lemma \ref{linsol} as in (\ref{firstordersol})
and $\theta_{ab}$ and $v$ to be determined.
We know that the $C^{3,\alpha,1}$ norms of $\eta$ and $u$
are bounded by $\norm{\mathcal{U}}$ times the norm of the data.
 
Now we need to eliminate the nonlinear errors defined by
  \begin{equation}
    \Qcal_{\Scal} \begin{bmatrix} \eta_{ab} \\ u \end{bmatrix}
    :=
    \Scal \begin{bmatrix} \delta + \eta \\ 1+u \end{bmatrix}
      - \Sdot \begin{bmatrix} \eta \\ u \end{bmatrix}
    \mbox{ and }
    \Qcal_{\Bcal}[\eta_{ab}]
    :=
    \Bcal[\delta + \eta] - \Bcal[\delta] - \Bdot[\eta].
  \end{equation}
that our first-order solutions introduce,
and accordingly we will make a standard application of the contraction mapping lemma
to secure a fixed point to the map
  \begin{equation}
    \label{contract}
    \begin{bmatrix} \theta_{ab} \\ v \end{bmatrix}
    \mapsto
    -\pi \Ucal 
      \begin{bmatrix}
        \Qcal_{\Scal}
          \begin{bmatrix} \eta_{ab}+\theta_{ab} \\ u+v \end{bmatrix} \\
        \Qcal_{\Bcal}[\eta_{ab}+\theta_{ab}]
      \end{bmatrix},
  \end{equation}
where $\pi$ projects onto the first two factors of the target of $\Ucal$
(forgetting the vector field $\chi$ used to modify the source $S$).

First we observe that the induced metric on $\partial M$
is independent of the potential $\Phi$
and linear in the ambient metric $g$,
while, for $g$ near $\delta$,
the mean curvature (being the divergence along $\partial M$
of the $g$-normalized vector field $g$-dual to the $1$-form $-dr$)
depends smoothly, as a function of some local coordinates on $\partial M$,
on the Cartesian components $g_{ij}$ of $g$
and on its first $g_{ij,k}$
partial derivatives.
Consequently we have
  \begin{equation}
    \norm{\Qcal_{\Bcal}[\eta_{ab}+\theta_{ab}]}_{2,\alpha}
    =
    \norm{\int_0^1 \int_0^t \frac{d^2}{ds^2} \Bcal[\delta+s(\eta_{ab}+\theta_{ab})] \, ds \, dt}_{2,\alpha}
    \leq
    C_1\norm{\eta+\theta}_{3,\alpha}^2 \mbox{ and }
  \end{equation}
  \begin{equation}
    \begin{aligned}
      \norm{\Qcal_{\Bcal}[\eta_{ab}+\theta_{ab}]-\Qcal_{\Bcal}[\eta_{ab}+\tilde{\theta}_{ab}]}_{2,\alpha}
      &=
      \norm{ \int_0^1 \int_0^t \int_0^1 \frac{\partial^3}{\partial s^2 \partial \tau}
        \Bcal[\delta+s(\eta+\tilde{\theta})+s\tau(\theta-\tilde{\theta})] \, d\tau \, ds \, dt}_{2,\alpha} \\
      &\leq
      C_1 \norm{\eta+\theta}_{3,\alpha} \norm{\theta-\tilde{\theta}}_{3,\alpha}
    \end{aligned}
  \end{equation}
for some $C_1>0$,
assuming $\norm{\theta-\tilde{\theta}}_{3,\alpha} \leq \norm{\eta+\theta}_{3,\alpha} < 1$.

We also observe, working in Cartesian coordinates,
that the components of $\Scal\left[\begin{smallmatrix}g \\ \Phi \end{smallmatrix}\right]$
are polynomial in $g_{ij}$, $g^{ij}$, $g_{ij,k}$, $g_{ij,k\ell}$,
$\Phi$, $\Phi_{,i}$, and $\Phi_{,ij}$,
so the components of $\Scal\left[\begin{smallmatrix}\delta+\eta \\ 1+u \end{smallmatrix}\right]$
depend analytically on $\eta_{ij}$,
$\eta_{ij,k}$, $\eta_{ij,k\ell}$,
$u$, $u_{,i}$, and $u_{,ij}$.
Moreover, each of the nonlinear terms in the corresponding power series
includes either (i) a product of two first derivatives of either $u$ or $\eta$
or (ii) a product of $u$ or $\eta$ with a second derivative of either $u$ or $\eta$.
Thus, estimating with the fundamental theorem of calculus as for $\Bcal$ above,
for some $C_2>0$ we have
  \begin{equation}
  \label{quaddecay}
    \norm{\Qcal_{\Scal}\begin{bmatrix} \eta_{ab}+\theta_{ab} 
      \\ u+v \end{bmatrix}}_{1,\alpha,4}
    \leq
    C_2 (\norm{\eta+\theta}_{3,\alpha,1}+\norm{u+v}_{3,\alpha,1})^2
  \end{equation}
and, assuming $\eta$, $u$, $\theta$, $\tilde{\theta}$, $v$, and $\tilde{v}$ sufficiently small,
  \begin{equation}
    \norm{\Qcal_{\Scal}\begin{bmatrix} \eta_{ab}+\theta_{ab} \\ u+v \end{bmatrix}
        - \Qcal_{\Scal}
          \begin{bmatrix} \eta_{ab}+\tilde{\theta}_{ab} 
            \\ u+\tilde{v} \end{bmatrix}}_{1,\alpha,4}
      \leq
      C_2 \left(\norm{\eta+\theta}_{3,\alpha,1}+\norm{u+v}_{3,\alpha,1}\right)
        \left(\norm{\theta-\tilde{\theta}}_{3,\alpha,1}+\norm{v-\tilde{v}}_{3,\alpha,1}\right).
  \end{equation}

Thus, by choosing $C_3>0$ sufficiently large and $\epsilon>0$
sufficiently small---both in terms of $C_1$, $C_2$, and the operator norm of $\Ucal$---and
by noting that $4>3\beta+1$, we see that
the map given by (\ref{contract}) is a contraction from
  \begin{equation}
  \label{domain}
    \left\{\norm{\theta_{ab}}_{3,\alpha,1}+\norm{v}_{3,\alpha,1}
    \leq 
    C_3\left( \norm{\gamma-\iota^*\delta}_{3,\alpha} + \norm{H+2}_{2,\alpha} \right)^2 \right\}
  \end{equation}
to itself.
By taking $(\theta_{ab},v)$ to be this map's unique fixed point and
  \begin{equation}
  \label{quadchi}
    \chi^a
    =
    \pi_3 \Ucal 
      \begin{bmatrix}
        \Qcal_{\Scal}
          \begin{bmatrix} \eta_{ab}+\theta_{ab} \\ u+v \end{bmatrix} \\
        \Qcal_{\Bcal}[\eta_{ab}+\theta_{ab}]
      \end{bmatrix},
  \end{equation}
where $\pi_3$ is projection onto the third factor of the target of $\Ucal$,
we obtain $(g,\Phi)$ solving
 \begin{equation}
    \Scal \begin{bmatrix} g_{ab} \\ \Phi \end{bmatrix}
    =
    \begin{bmatrix} \chi_{a;b}+\chi_{b;a} \\ 0 \end{bmatrix}
    \mbox{ and }
    \Bcal[g_{ab}]
    =
    \begin{bmatrix} \gamma_{ab} \\ H \end{bmatrix}.
  \end{equation}
Since $\theta$ is a component of a fixed point of (\ref{contract})
on the domain (\ref{domain}),
in particular we have
  \begin{equation}
  \label{thetaest}
    \norm{\theta_{ab}}_{3,\alpha,1}
    \leq 
    C_3\left( \norm{\gamma-\iota^*\delta}_{3,\alpha} + \norm{H+2}_{2,\alpha} \right)^2.
  \end{equation}

As planned we take $g=\delta+\eta+\theta$ and $\Phi=1+u+v$,
so that for $\epsilon$ sufficiently small
$g$ is a (nondegenerate) metric and $\Phi>0$.
Writing $R_{ab}$ and $R$ for the Ricci and scalar curvature of $g$,
we have now achieved
  \begin{equation}
  \label{almostsv}
  \begin{aligned}
    \Phi_{|ab} - \Phi R_{ab} = \chi_{a;b} + \chi_{b;a} \\
    g^{cd}\Phi_{|cd} = 0,
  \end{aligned}
  \end{equation}
where $g^{cd}$ denotes $g^{-1}$ (rather than contraction of $g$ with $\delta^{-1}$)
and each vertical bar
indicates covariant differentiation relative to $g$
(and each semicolon continues to indicate
covariant differentiation relative to $\delta$).
Mimicking \cite{MiaoStatExt},
we derive from (\ref{almostsv}) a linear elliptic system for $\chi$, with coefficients
depending on $g$ and $\Phi$.
For brevity we set
  \begin{equation}
    \tau_{ab}=\chi_{a;b}+\chi_{b;a} \in C^{1,\alpha,3+\beta}\left(T^*M^{\odot 2}\right),
  \end{equation}
the right-hand side of the first equation in (\ref{almostsv}).
Then, taking the $g$ trace of the first equation,
using the second, rearranging, and taking a derivative, we get
  \begin{equation}
    R_{|a}=-\Phi^{-1}g^{cd}\tau_{cd|a} + \Phi_{|a}\Phi^{-2}g^{cd}\tau_{cd},
  \end{equation}
while the $g$ divergence of the first gives
  \begin{equation}
    -\Phi R_{ab|c}g^{bc} = \tau_{ab|c}g^{bc}.
  \end{equation}

Combining these last two equations via the twice contracted Bianchi identity we arrive at
  \begin{equation}
    \tau_{ab|c}g^{bc}
      -\frac{1}{2}g^{cd}\tau_{cd|a}
      +\frac{1}{2}\left(\ln \Phi\right)_{|a}g^{cd}\tau_{cd}
    =
    0,
  \end{equation}
which we rewrite as
  \begin{equation}
    \chi_{a;b}^{\;\;\;\;\; ;b}
    =
    B^b\tau_{ab}+C^{cd}_{\;\;\;\; a}\tau_{cd}+D^{cd}\tau_{cd;a}
  \end{equation}
for tensor fields on $M$ satisfying 
$\norm{B}_{2,\alpha,2}+\norm{C}_{2,\alpha,2}+\norm{D}_{3,\alpha,1} \leq C_4 \epsilon$,
for some $C_4>0$ independent of $\epsilon$,
but the estimate for $\chi$ in Lemma \ref{linsol}
then implies the bound
  \begin{equation}
    \norm{\chi}_{2,\alpha,2+\beta} \leq C_5 \epsilon \norm{\chi}_{2,\alpha,2+\beta}
  \end{equation}
for some $C_5>0$ independent of $\epsilon$, which
for $\epsilon$ small enough forces $\chi=0$.
Thus $g$ is exactly static with nowhere vanishing potential $\Phi$.
The estimates of $g_{ab}$ and $\Phi$ follow from the above bounds
for $\eta_{ab}$, $u$, $\theta_{ab}$, and $v$, completing the proof of (i).

\textbf{(ii)}
The uniqueness can be established by a contradiction argument,
appealing again to the contractiveness of the nonlinear terms established above
and then studying the linearized problem with trivial data.
Instead of carrying out this approach,
now that we have in hand a solution with estimates adequate for our application,
we refer the reader to \cite{AndLoc} for a proof of uniqueness,
which will not be needed in this article.

\textbf{(iii)}
For the analyticity see \cite{MzH} or Proposition 2.8 in \cite{CorDef}.
A proof of the existence of a harmonic coordinate system near infinity
for an arbitrary asymptotically flat metric can be found in \cite{BarADM} for instance.
In the present, nearly-Euclidean setting,
if $\{x^i\}$ are Cartesian coordinates on $\R^3$ (restricted to $M$),
the estimate for $g-\delta$ ensures that
$\Delta_g x^i$ can be made arbtirarily small in $C^{0,\alpha,2+\beta}$
by taking $\epsilon$ small, so, picking a bounded right inverse
$\widetilde{G}: C^{0,\alpha,2+\beta} \to C^{2,\alpha,\beta}$ for $\Delta_g$,
we get $g$-harmonic coordinates $\{x^i-\widetilde{G}\Delta_g x^i\}$ on the interior of $M$.

\textbf{(iv)} The absence of closed minimal surfaces likewise follows from the smallness of $g-\delta$
and from the maximum principle.
Indeed the function $r^2$ on $M$ of course has $\delta$-Hessian $(r^2)_{;ab}=2\delta_{ab}$
and $g$-Hessian $(r^2)_{|ab}=(r^2)_{;ab}-r(g_{ar;b}+g_{br;a}-g_{ab;r})$.
For $\epsilon$ sufficiently small $r^2$ is then everywhere strictly convex on $(M,g)$,
so its restriction to any minimal surface in $(M,g)$
is subharmonic and as such can attain a maximum value only on its boundary.
Thus $M$ is devoid of closed minimal surfaces,
and similarly if $\Sigma$ is any least-area surface in $M$ enclosing $\partial M$,
then off $\partial M$ it is a properly embedded minimal surface,
so in fact coincides with $\partial M$, which is thereby outer-minimizing.

\textbf{(v)}
For the mass estimate we first observe that
  \begin{equation}
    \madm[g]=\frac{1}{16\pi} \lim_{r \to \infty} 
      \int_{\abs{\vec{x}}=r} \left( \eta_{ra}^{\;\;\;\; ;a} - \eta^c_{\;\; cr} \right)
        + O\left(\norm{\gamma-\iota^*\delta}_{3,\alpha}^2+\norm{H+2}_{2,\alpha}^2\right)
  \end{equation}
in light of the bound (\ref{thetaest}).
By item (ii) of Lemma \ref{linsol} the linearized metric $\eta$
is conformally flat at infinity with harmonic conformal factor $-2u$,
so that, using equation (\ref{uavg}), for large $r$
  \begin{equation}
    \begin{aligned}
      \int_{\abs{\vec{x}}=r} \left( \eta_{ra}^{\;\;\;\; ;a} - \eta^c_{\;\; cr} \right)
      &=
      4\int_{\abs{\vec{x}}=r} u_{,r}
      =
      4\int_{\abs{\vec{x}}=1} u_{,r}
      =
      -4\int_{\partial M} u \\
      &=
      \int_{\partial M} \left[2(H+2) - (\gamma-\iota^*\delta)_{\mu}^{\;\; \mu}\right].
    \end{aligned}
  \end{equation}

\end{proof}

\section{Proof of the main theorem and its corollary}
\subsection*{Proof of Theorem \ref{mouthm}}
To prove Theorem \ref{mouthm} take $\epsilon$ as in Proposition \ref{statext}
(say with $\alpha=\beta=1/2$),
$\gamma \in C^4(T^*\partial M^{\odot 2}), H \in C^3(\partial M)$ 
satisfying
$\norm{\gamma-\iota^*\delta}_{C^4}+\norm{H+2}_{C^3}<\epsilon$,
and let $g$ be the corresponding extension guaranteed by item (i) of the proposition.
Then $g \in \mathcal{PM}[\gamma,H]$ (recall (\ref{PM})),
and by item (iv) of the proposition $\partial M$ is outer-minimizing in $(M,g)$,
so by definition (\ref{mout}) the ADM mass of $g$ is an upper bound for
the outer mass of $(S^2,\gamma,H)$, whence by item (v) of the proposition
  \begin{equation}
  \label{ub}
    \mout[\gamma,H]
    \leq 
    \frac{1}{16\pi}\int_{S^2} (6+2H-\gamma^\mu_{\;\; \mu})
      + C\left(\norm{\gamma_{\mu \nu}-(\iota^*\delta)_{\mu \nu}}_{3,\alpha}+\norm{H+2}_{2,\alpha}\right)^2.
  \end{equation}

In contrast to the elementary calculations leading to the upper bound,
for the lower bound we appeal to the inverse mean curvature flow
of Huisken and Ilmanen.
It is a well-known consequence of their proof \cite{HI}
of the Riemannian Penrose Inequality
(see in particular the remarks on page 426 preceding the proof of Positivity Property 9.1) 
that the ADM mass of an asymptotically flat manifold
satisfying the decay conditions (\ref{afdecay}) (or somewhat weaker conditions) and
having nonnegative scalar curvature is no less than the Hawking mass $\mH$ of any outer-minimizing
sphere it contains. Thus
  \begin{equation}
  \label{lb}
    \mout[\gamma,H] 
    \geq \mH[\gamma,H]
    :=
    \sqrt{\frac{\int_{S^2} \sqrt{\abs{\gamma}}}{16\pi}}
      \left(1-\frac{1}{16\pi}\int_{S^2} H^2 \, \sqrt{\abs{\gamma}}\right),
  \end{equation}
where $\sqrt{\abs{\gamma}}$
denotes the integration measure induced on $S^2$ by $\gamma$.
On the other hand a quick calculation reveals that
the linearization of $\mH$ at $(\iota^*\delta,-2)$
is just the integral appearing in (\ref{ub}),
which completes the proof of the theorem.

\subsection*{Proof of Corollary \ref{sscor}}

To prove the corollary let $p$ be a point in a smooth Riemannian manifold $(N,g)$,
and for each small $\tau>0$ let $B_\tau$ be the geodesic ball of center $p$ and radius $\tau$.
Setting $\gamma=\iota^*\tau^{-2}g$ and $H=\Hcal_\iota[\tau^{-2}g]$
and writing $R_{ab}$ for the Ricci curvature of $g$ at $p$
and $x^i$ for the i\textsuperscript{th} Cartesian coordinate function on $\R^3$
restricted to $S^2$,
by Taylor expansion in $\tau$ of $\gamma$ and $H$ (included in Appendix \ref{sms}),
we compute
  \begin{equation}
    \begin{aligned}
      \frac{1}{16\pi}\int_{S^2} \left(6+2H-\gamma^\mu_\mu\right)
      &=
      \frac{1}{16\pi}\int_{S^2} \left(\tau^2 R_{ij}x^ix^j+\frac{2}{3}\tau^3 R_{ij|k}x^ix^jx^k
        +\frac{1}{4}\tau^4 R_{ij|k\ell}x^ix^jx^kx^\ell \right) + O(\tau^5) \\
      &=
      \frac{1}{48\pi}\tau^2 R_{ij} \int_{S^2} \abs{x}^2\delta^{ij}
        + 0 + \frac{1}{64\pi}\tau^4 R_{ij|k\ell} \int_B \partial^i (x^jx^kx^\ell) + O(\tau^5) \\
      &=
      \frac{1}{12}R\tau^2 + \frac{1}{192\pi}\tau^4R_{ij|k\ell} \int_B \abs{x}^2
        \left(\delta^{ij}\delta^{k\ell}+\delta^{ik}\delta^{j\ell}+\delta^{i\ell}\delta^{jk}\right)
        + O(\tau^5) \\
      &=
      \frac{1}{12}R\tau^2 + \frac{1}{120}\Delta R\tau^4 + O(\tau^5)
    \end{aligned}
  \end{equation}
and
  \begin{equation}
    \norm{\gamma-\iota^*\delta}_{C^4}^2+\norm{H+2}_{C^3}^2
    \leq
    R_{ijk\ell}R^{ijk\ell}\tau^4 + O(\tau^5).
  \end{equation}
It follows from the theorem that
  \begin{equation}
    \mout[B_\tau, \tau^{-2}g] = 
      \begin{cases}
        \frac{1}{12}R\tau^2 + O(\tau^4) \mbox{ in general} \\
        \frac{1}{120}\Delta R \tau^4 + O(\tau^5) \mbox{ if $g$ is flat at $p$}.
      \end{cases}
  \end{equation}
Now the scaling law
$\madm[M,\lambda^2h]=\lambda \madm[M,h]$
for the ADM mass implies $\mout[B_\tau,g]=\tau\mout[B_\tau, \tau^{-2}g]$, 
completing the proof of the corollary.

\appendix

\section{Rapidly decaying solutions to the Poisson equation}
\label{Peqn}

The following result is needed in the proof of Lemma \ref{inhomlinsol}.
In this appendix $\Delta=\Delta_\delta$ is simply the flat Laplacian on $M$. 
Note that on $\R^3$ an estimate like (\ref{Pest})
for the Poisson equation $\Delta u = f$
is impossible whenever $q \geq 1$.
For example a strictly positive, smooth, compactly supported, spherically symmetric source $f$
obviously exhibits arbitrarily fast decay,
but the unique bounded solution is a nonzero constant times $r^{-1}$ outside a compact set.
On $M$ we can do better by forbidding the low (compared to $q$)
spherical harmonics at infinity.
 
\begin{lemma}
\label{Pfd}
Let $\alpha,\beta \in (0,1)$ and let $k,q$ be nonnegative integers.
Then there is a constant $C(\alpha,\beta,k,q)>0$ such that
for every $f \in C^{k,\alpha,q+\beta+2}(M)$
there exists $u \in C^{k+2,\alpha,q+\beta}(M)$
solving the Poisson equation $\Delta u = f$ and satisfying the bound
  \begin{equation}
  \label{Pest}
    \norm{u}_{k+2,\alpha,q+\beta} \leq C(\alpha,\beta,k,q)\norm{f}_{k,\alpha,q+\beta+2}.
  \end{equation}
\end{lemma}

\begin{proof}
Fix a bounded extension operator $E: C^{k,\alpha,q+\beta+2}(M) \to C^{k,\alpha,q+\beta+2}(\R^3)$,
so that $Ef|_M=f$ for every $f$ in the domain of $E$.
For each $\vec{x} \in \R^3 \backslash \{0\}$ let $T_{q-1}(\vec{x},\cdot)$
be the degree-$(q-1)$ Taylor polynomial centered at $\vec{x}$
for the Newton potential $\Gamma(\vec{x})=-\abs{4\pi \vec{x}}^{-1}$ on $\R^3$
and let
$R_{q-1}(\vec{x},\cdot)=\Gamma(\vec{x}+\cdot)-T_{q-1}(\vec{x},\cdot)$
be the corresponding remainder.
Then, since $\Gamma$ is harmonic on $\R^3 \backslash \{0\}$,
for each $\vec{y} \in \R^3$ each term of $T_{q-1}(\cdot,\vec{y})$
is also harmonic on $\R^3 \backslash \{0\}$,
and therefore, $\Gamma$ itself being a fundamental solution for the Poisson equation,
the convolution
  \begin{equation}
  \label{udef}
    u(\vec{x})=\int_{\R^3} R_{q-1}(\vec{x},-\vec{y})(Ef)(\vec{y}) d^3y    
  \end{equation}
defines, at least distributionally, a solution to $\Delta u = f$ on $M$.

Given $\vec{x} \in \left\{\abs{\vec{x}} \geq \frac{1}{2} \right\}$,
we now estimate $\abs{u(x)}$,
assuming $f \in C^{0,0,q+\beta+2}(M)$, as follows.
On $\{\abs{\vec{y}} \geq 2\abs{\vec{x}}\} \subset \R^3$
we use $\abs{\Gamma(\vec{x}-\vec{y})} \leq 2\abs{\vec{y}}^{-1}$ to get
  \begin{equation}
  \label{bigyG}
  \abs{\int_{\{\abs{\vec{y}} \geq 2\abs{\vec{x}}\}} \Gamma(\vec{x}-\vec{y})(Ef)(\vec{y}) \, d^3y}
   \leq
  \frac{8\pi}{q+\beta}2^{-q-\beta}\norm{f}_{0,0,q+\beta+2}\abs{\vec{x}}^{-q-\beta},
  \end{equation}
while on $\left\{\frac{1}{2}\abs{\vec{x}} \leq \abs{\vec{y}} \leq 2\abs{\vec{x}} \right\}
\subset \R^3$
we use $\abs{\vec{x}-\vec{y}} \leq 4\abs{\vec{x}}$ and
$\abs{(Ef)(\vec{y})} \leq \left(1+\frac{1}{2}\abs{\vec{x}}\right)^{-q-\beta-2}$
to get
  \begin{equation}
  \label{middleyG}
    \abs{\int_{\left\{\frac{1}{2}\abs{\vec{x}} \leq \abs{\vec{y}} \leq 2\abs{\vec{x}} \right\}}
      \Gamma(\vec{x}-\vec{y})(Ef)(\vec{y}) \, d^3y}
    \leq
    2^{q+\beta+5}\norm{f}_{0,0,q+\beta+2}\abs{\vec{x}}^{-q-\beta}.
  \end{equation}
On the other hand, each term of $T_{q-1}(\vec{x},-\vec{y})$
with total degree $j$ in $\vec{y}$
has absolute value bounded by $C(j)\abs{\vec{x}}^{-j-1}\abs{\vec{y}}^j$,
so
  \begin{equation}
  \label{bigyT}
    \abs{\int_{\left\{\abs{\vec{y}} \geq \frac{1}{2}\abs{\vec{x}}\right\}}
        T_{q-1}(\vec{x},-\vec{y})(Ef)(\vec{y}) \, d^3y}
    \leq
    C(q,\beta)\norm{f}_{0,0,q+\beta+2}\abs{\vec{x}}^{-q-\beta},
  \end{equation}
and, last, on the region $\left\{\abs{\vec{y}} \leq \frac{1}{2}\abs{\vec{x}} \right\}$
the remainder estimate
$\abs{R_{q-1}(\vec{x},-\vec{y})} \leq C(q) \abs{\vec{x}}^{-q-1} \abs{\vec{y}}^q$
yields
  \begin{equation}
  \label{smally}
    \abs{\int_{\left\{\abs{\vec{y}} \leq \frac{1}{2}\abs{\vec{x}} \right\}}
      R_{q-1}(\vec{x},-\vec{y})(Ef)(\vec{y}) \, d^3y}
    \leq
    C(q,\beta)\norm{f}_{0,0,q+\beta+2} \abs{\vec{x}}^{-q-\beta}.
  \end{equation}
Combining (\ref{udef}), (\ref{bigyG}), (\ref{middleyG}), (\ref{bigyT}), and (\ref{smally}),
we get
  \begin{equation}
  \label{C0est}
    \norm{u}_{0,0,q+\beta} \leq C(q,\beta) \norm{f}_{0,0,q+\beta+2}.
  \end{equation}

Next we have the weighted Schauder estimate
that there exists some $C(k,q,\alpha,\beta)>0$ such that for any
$v \in C^{k+2,\alpha,q+\beta}\left(\left\{\abs{\vec{x}} \geq \frac{1}{2}\right\}\right)$ we have
  \begin{equation}
  \label{Schauder}
    \norm{v|_M}_{k+2,\alpha,q+\beta}
    \leq
    C(\alpha,\beta,k,q)
    \left(
      \norm{\Delta v}_{k,\alpha,q+\beta+2}+\norm{v}_{0,0,q+\beta}
    \right),
  \end{equation}
which can be established 
by applying a scaling argument to the standard interior Schauder estimate,
as in the proof of Proposition 26 in the appendix of \cite{SmiWei}.
In fact that very proof (which refers to $\R^3$ in its entirety)
applies here with obvious minor modifications.
The estimate (\ref{Pest}) then follows immediately from 
(\ref{Schauder}) and (\ref{C0est}), provided we know
that $u \in C^{k+2,\alpha,q+\beta}\left(\left\{\abs{\vec{x}} \geq \frac{1}{2}\right\}\right)$.

To see this last inclusion suppose $f \in C^{k,\alpha,q+\beta+2}(M)$.
Using cut-off functions and mollifiers
we can construct a sequence $\{f_n\}_{n=1}^\infty$
of functions, each defined and smooth on all of $\R^3$, such that
(i) for each $n$ the function $f_n$ is supported in $\{\abs{\vec{x}}<n\}$,
(ii) $\{f_n\}$ is bounded in $C^{k,\alpha,q+\beta+2}\left(\R^3\right)$
 by a constant times $\norm{f}_{k,\alpha,q+\beta+2}$,
and (iii) $\{f_n\}$ converges in $C^{k,0,q+2}\left(\R^3\right)$ to $Ef$.
For each nonnegative integer $n$ we next define the function
$u_n$ on $\R^3 \backslash \{0\}$ by
$u_n(\vec{x})=\int_{\R^3} R_{q-1}(\vec{x},-\vec{y}) f_n(\vec{y}) \, d^3y$.
Then by elliptic regularity each $u_n$ is smooth on $\R^3 \backslash \{0\}$,
where it is a classical solution of $\Delta u_n = f_n$.
Moreover $u_n$ is harmonic outside $\{\abs{\vec{x}}<n\}$.
Taking $\vec{x}$ with $\abs{\vec{x}} \geq 2n$
we can apply standard interior estimates for harmonic functions
(for example Theorem 2.10 in \cite{GT})
on a ball of center $\vec{x}$ and radius $\frac{1}{2}\abs{\vec{x}}$
in conjunction with (\ref{C0est}) to secure
$u_n \in C^{k+2,\alpha,q+\beta}(M)$ as claimed.
Thus, as shown above, (\ref{Pest}) applies to each $u_n$,
establishing that $\{u_n\}$ is bounded in $C^{k+2,\alpha,q+\beta}(M)$
by $C(k,q,\alpha,\beta)\norm{f}_{k,\alpha,q+\beta+2}$.
Consequently (see for example Lemma 27 in \cite{SmiWei})
there exists a subsequence converging in $C^{k+2,0,q}(M)$
to a function $u \in C^{k+2,\alpha,q+\beta}(M)$
necessarily satisfying
$\norm{u}_{k+2,\alpha,q+\beta} \leq C(\alpha,\beta,k,q)\norm{f}_{k,\alpha,q+\beta+2}$
and $\Delta u=f$,
which ends the proof.
\end{proof}

\section{Variation of the boundary data with respect to ambient metric}
\label{varapp}
Let $M$ be a manifold equipped with a smooth family
$\{g_t\}_{t \in \R}$ of Riemannian metrics
(with the $t$ subscript frequently suppressed below),
and let $\phi: \Sigma \to M$ be a codimension-one immersion,
which we may assume to be a two-sided embedding,
since the following calculations are purely local.
We use the abstract index notation
(as described in Section 2.4 of \cite{Wald} for example),
labelling components in $TM$ (or its dual) with Latin indices
and components in $T\Sigma$ (or its dual) with Greek indices.

For notational convenience we set
  \begin{equation}
    \dot{g}_{ab}:=\frac{d}{dt}\left(g_{ab}\right).
  \end{equation}
As is typical we write $g^{-1}$ as $g^{ab}$ and,
at any given value of $t$, indices are raised and lowered via $g_t$.
Thus, to clarify our notation,
from the identity $g_{ac}g^{cb}=\delta_a^{\;\;b}$ and the product rule we have for instance
  \begin{equation}
    \frac{d}{dt}\left(g^{ab}\right)=-\dot{g}^{ab}.
  \end{equation}
We also fix a ($t$-dependent) unit normal $N^a$ and write $\nu^a$ for its value at $t=0$.
An $n$ index on a tensor will indicate contraction with $N$;
for example
  \begin{equation}
    \dot{g}_{nn}=\dot{g}_{ab}N^aN^b=\dot{g}^{ab}N_aN_b=\dot{g}^{nn}
=\dot{g}^n_{\;\;n}
=\dot{g}_n^{\;\;n}.
  \end{equation}

Setting
  \begin{equation}
    \dot{N}_a:=\frac{d}{dt}\left(N_a\right),
  \end{equation}
since the $1$-form $N_a$ is always proportional to $\nu_a$ and since
  \begin{equation}
    0=\frac{d}{dt}(g^{ab}N_aN_b)=-\dot{g}^{nn}+2\dot{N}_aN^a,
  \end{equation}
we have
  \begin{equation}
    \dot{N}_a = \frac{1}{2}\dot{g}^{bc}N_aN_bN_c = \frac{1}{2}\dot{g}^{nn}N_a.
  \end{equation}

We write $A_{\alpha \beta}$ for the ($t$-dependent)
scalar-valued second fundamental form of $\phi$
relative to $g_{ab}$ and $N^a$.
We extend $N^a$ (a section of $\phi^*(TM)$)
by ($t$-dependent) parallel transport along the ($t$-dependent)
geodesics it generates to a vector field (a section of $TM$) of the same name
on a neighborhood of $\Sigma$.
Then
  \begin{equation}
    A_{\alpha \beta}=-\phi^a_{\;\;,\alpha}\phi^b_{\;\;,\beta}N_{a|b},
  \end{equation}
using a vertical bar to indicate covariant differentiation defined by $g_t$.
Where convenient we will alternatively use $\Dbar$
to represent the same Levi-Civita connection induced by $g_t$,
and we define $\dot{\Dbar}_{ab}^{\;\;\;\;c}$ by
  \begin{equation}
    \dot{\Dbar}(X,Y):=\frac{d}{dt}\Dbar_X Y \mbox{ for $X^a,Y^a$ independent of $t$}.
  \end{equation}

Now we compute
  \begin{equation}
    \begin{aligned}
      \frac{d}{dt}\left(N_{a|b}\right)&=-\dot{\Dbar}_{abc}N^c+\Dbar_b \dot{N}_a \\
      &=
      -\frac{1}{2}\left(\dot{g}_{ac|b}+\dot{g}_{bc|a}-\dot{g}_{ab|c}\right)N^c
        + \frac{1}{2}\dot{g}^{nn}N_{a|b} + \frac{1}{2}\dot{g}^{nn}_{\;\;\;\; ,b}N_a,
    \end{aligned}
  \end{equation}
so
  \begin{equation}
    \dot{A}_{\alpha \beta}
    :=
    \frac{d}{dt}\left(A_{\alpha \beta}\right)
    =
    \frac{1}{2}A_{\alpha \beta}\dot{g}^{nn}
      +\frac{1}{2}\phi^a_{\;\;,\alpha}\phi^b_{\;\;,\beta}
        \left(\dot{g}_{an|b}+\dot{g}_{bn|a}-\dot{g}_{ab|n}\right)
  \end{equation}
and therefore, writing $H:=\left(\phi^*g\right)^{\alpha \beta}A_{\alpha \beta}$
for the scalar-valued mean curvature of $\phi$ relative to $g_t$ and $N_t$
(and raising Greek indices via $\phi^*g_t$),
  \begin{equation}
  \label{Hvar}
    \dot{H}=-\phi^a_{\;\;,\alpha}\phi^b_{\;\;,\beta}\dot{g}_{ab} A^{\alpha \beta}
      + \frac{1}{2}H\dot{g}^{nn}
      + \frac{1}{2}\phi^a_{\;\;,\alpha}\phi^{b,\alpha}
        \left(\dot{g}_{an|b}+\dot{g}_{bn|a}-\dot{g}_{ab|n}\right).
  \end{equation}

\subsection*{Diffeomorphisms}
Now suppose $g_0$ is a given metric on $M$ and suppose $\nu^a$ is a unit normal as above,
extended, without relabelling, to a vector field on a neighborhood of $\Sigma$
by parallel translation along the geodesics it generates.
We consider $g_t$ defined by the evolution of $g_0$ under the flow generated by
a $t$-independent vector field $\xi^a$.

\subsubsection*{Normal fields}
First suppose $\xi^a=f\nu^a$ for some $t$-independent function $f \in C^2_{loc}(M)$.
Then
  \begin{equation}
    \dot{g}_{ab}=\left(L_{\xi}g_t\right)_{ab}=f_{|b}\nu_a+f_{|a}\nu_b+2f\nu_{a|b},
  \end{equation}
where $L$ denotes Lie differentiation and we have used the fact that $\nu_{a|b}=\nu_{b|a}$.
Consequently
  \begin{equation}
    \frac{d}{dt}(\phi^*g)_{\alpha \beta}=-2A_{\alpha \beta}\phi^*f,
  \end{equation}
  \begin{equation}
    \dot{g}^{nn}=2f_{,n}
  \end{equation}
and
  \begin{equation}
    \dot{g}_{ab|c}=f_{|bc}\nu_a+f_{|b}\nu_{a|c}
      +f_{|ac}\nu_b+f_{|a}\nu_{b|c}+2f_{|c}\nu_{a|b}+2f\nu_{a|bc}.
  \end{equation}
In particular we find
  \begin{equation}
    \begin{aligned}
      \phi^a_{\;\;,\alpha}\phi^b_{\;\;,\beta}\dot{g}_{ab|n}
      &=
      -2(\phi^*f_{,n})A_{\alpha \beta}+2\phi^a_{\;\;,\alpha}\phi^b_{\;\;,\beta}\nu_{a|bn} \\
      &=
      -2(\phi^*f_{,n})A_{\alpha \beta}+2\phi^a_{\;\;,\alpha}\phi^b_{\;\;,\beta}\nu_{a|nb}
        +2\phi^a_{\;\;,\alpha}\phi^b_{\;\;,\beta}R_{nbna},
    \end{aligned}
  \end{equation}
with curvature convention
$R_{abcd}=\langle \Dbar_{\partial_a}\Dbar_{\partial_b}\partial_c
  -\Dbar_{\partial_b}\Dbar_{\partial_a}\partial_c, \partial_d \rangle$,
and
  \begin{equation}
    \phi^a_{\;\;,\alpha}\phi^b_{\;\;,\beta}\dot{g}_{an|b}
    =
    -(\phi^*f_{,n})A_{\alpha \beta}+\phi^a_{\;\;,\alpha}\phi^b_{\;\;,\beta}f_{|ab}.
      +2\phi^a_{\;\;,\alpha}\phi^b_{\;\;,\beta}f\nu_{a|nb},
  \end{equation}
Using also
  \begin{equation}
    \phi^a_{\;\;,\alpha}\phi^b_{\;\;,\beta}f_{|ab}
    =
    (\phi^*f)_{:\alpha \beta} - A_{\alpha \beta}\phi^*f_{,\nu},
  \end{equation}
with $:$ indicating covariant differentiation relative to $\phi^*g$,
and
  \begin{equation}
    \nu_{a|nb}=\nu_{n|ab}=-\nu_{a|c}\nu^c_{\;\;b} 
  \end{equation}
at last we obtain from (\ref{Hvar})
  \begin{equation}
    \dot{H}= (\phi^*f)_{:\alpha}^{\;\;\; :\alpha}+\abs{A}^2\phi^*f+\phi^*R_{nn}f.
  \end{equation}

\subsubsection*{Tangential fields}
Now suppose $\xi^a=W^a$ for $W \in C^2_{loc}(TM)$ everywhere orthogonal to $\nu$.
Then $W$ restricts to a vector field on $\Sigma$ which we will also call $W$. 
Then
  \begin{equation}
    \dot{g}_{ab}=\left(L_{_W}g_t\right)_{ab}=W_{a|b}+W_{b|a},
  \end{equation}
so
  \begin{equation}
    \begin{aligned}
      &\frac{d}{dt}(\phi^*g)_{\alpha \beta}=W_{\alpha:\beta}+W_{\beta:\alpha}, \\
      &\dot{g}^{nn}=0, \\
      &\phi^a_{\;\;,\alpha}\phi^b_{\;\;,\beta}\dot{g}_{ab|n}
        = \phi^a_{\;\;,\alpha}\phi^b_{\;\;,\beta}(W_{a|bn}+W_{b|an}), \mbox{ and} \\
      &\phi^a_{\;\;,\alpha}\phi^b_{\;\;,\beta}\dot{g}_{an|b}
        = \phi^a_{\;\;,\alpha}\phi^b_{\;\;,\beta}(W_{a|nb}+W_{n|ab}).
    \end{aligned}
  \end{equation}
Therefore
  \begin{equation}
    \begin{aligned}
      \frac{1}{2}\phi^a_{\;\;,\alpha}\phi^{b_,\alpha}(\dot{g}_{an|b}+\dot{g}_{bn|a}-\dot{g}_{ab|n})
      &=
      \phi^a_{\;\;,\alpha}\phi^{b_,\alpha}(W_{a|nb}+W_{n|ab}-W_{a|bn}) \\
      &=
      \phi^a_{\;\;,\alpha}\phi^{b_,\alpha}W_{n|ab}+R_{nc}W^c,
    \end{aligned}
  \end{equation}
but the first term in (\ref{Hvar}) is
  \begin{equation}
    \begin{aligned}
      -\phi^a_{\;\;,\alpha}\phi^{b_,\alpha}\dot{g}_{ab} A^{\alpha \beta}
      &=
      2W_{a|b}N^{a|b} \\
      &=
      (W_aN^a)^b_{\;\;|b}-W_{a|b}^{\;\;\;\;\; |b}N^a-W_aN^{a|b}_{\;\;\;\;\; |b} \\
      &=
      0 - W_{n|c}^{\;\;\;\;\; |c} - W^aN^b_{\;\; |ab} \\
      &= -W_{n|c}^{\;\;\;\;\; |c} - R_{an}W^a + W^cH_{|c},
    \end{aligned}
  \end{equation}
and therefore, noting $W_{n|n}^{\;\;\;\;\; |n}=0$, we arrive at
  \begin{equation}
    \dot{H}=WH.
  \end{equation}

\subsection*{Conformal change}
Of course we can also apply (\ref{Hvar}) to
linearized conformal transformations,
but the noninfinitesimal transformation laws are simple enough.
If $h_{ab}$ is a fixed metric on $M$ and $\rho \in C^1_{loc}(M)$ is everywhere strictly positive,
then under the conformal change $\tilde{h}_{ab}=\rho^2h_{ab}$ we have the corresponding
conformal change $\phi^*\tilde{h}=(\phi^*\rho^2)(\phi^*h)$ for the induced metric on $\Sigma$.
For the change of its second fundamental form from $A$ to $\tilde{A}$, starting from the identity
$A=-\frac{1}{2}\phi^*L_{_N}h$ we compute
  \begin{equation}
    \begin{aligned}
      \tilde{A}
      &=
      -\frac{1}{2}\phi^*L_{\rho^{-1}N}(\rho^2h) \\
      &=
      -\frac{1}{2}(\phi^*\rho^{-1})\phi^*L_{_N}(\rho^2h) \mbox{ (since $N \perp \phi_*T\Sigma$)} \\
      &=
      (\phi^*\rho)A - (\phi^*\rho_{,n})\phi^*h,
    \end{aligned}
  \end{equation}
whence
  \begin{equation}
    \tilde{H} = (\phi^*\rho^{-1})H - (\dim \Sigma)(\phi^*\rho^{-2} \rho_{,n}).
  \end{equation}

Thus for a conformal family of metrics $g_t=\rho^2(t)g_0$ with $\rho(0)=1$
  \begin{equation}
    \begin{aligned}
      \dot{H}_0&=-(\phi^*\dot{\rho})H_0 - (\dim \Sigma)(\phi^*\dot{\rho}_{,n}) \\
        &= -\frac{H}{2}\left.\frac{d}{dt}\right|_{t=0}\phi^*\rho^2
            - \frac{\dim \Sigma}{2}\left.\frac{d}{dt}\right|_{t=0}\phi^*(\rho^2)_{,n}.
    \end{aligned}
  \end{equation}

\section{The induced metric of small metric spheres}
\label{sms}

Fix a point $p$ in a Riemannian manifold $(M,g)$ of dimension $n+1$,
and let $S^n$ be the unit $n$-sphere centered at the origin in $T_pM$.
Define $\Phi: \R \times S^n \to M$ by
  \begin{equation}
    \Phi(t,\theta)=\exp_p t\theta,
  \end{equation}
where $\exp$ is the exponential map corresponding to $g$.
We will casually identify vector fields on $S^n$
with $T_pM$-valued maps on $S^n$ (whose preimages
are orthogonal to their images)
as well as with their $t$-independent extensions to $\R \times S^n$.
Writing $\overline{D}$ for the Levi-Civita connection on $TM$ determined by $g$,
there is a unique connection $D$ on $\Phi^*TM$
satisfying $D_X \Phi^*\xi = \overline{D}_{\Phi_*X}\xi$
for any $X \in C^0_{loc}(T(\R \times S^n))$,
$\xi \in C^1_{loc}(TM)$;
it is torsion-free in the sense that
$D_X \Phi_*Y - D_Y \Phi_*X = \Phi_*[X,Y]$
and metric-compatible.
Writing $P_s^t$ for the corresponding parallel-transport map
from the fiber over $(s,\cdot)$ to the fiber over $(t,\cdot)$,
define further the time-dependent map
$\phi(t)$ from $TS^n$ to $T_pM$ by
  \begin{equation}
    \phi(t)V=P_t^0 \Phi_*V.
  \end{equation}
Defining also the parallely transported curvature operator
$\Rbar^a_{\;\; rbr}$ by
$\Rbar^a_{\;\; rbr}V^b=P_t^0 R(\Phi_*\partial_t,V)\Phi_*\partial_t$
and with similar notation for covariant derivatives of curvature,
we have 
  \begin{equation}
    \begin{aligned}
      &\phi(0)=0, \\
      &\partial_t\phi(0)V=\left.D_V\Phi_*\partial_t\right|_{t=0}=V, \\
      &\partial_t^2\phi^a_{\;\; b}(t)=\Rbar^a_{\;\; rcr}\phi^c_{\;\; b}, \\
      &\partial_t^3\phi^a_{\;\; b}(t)=\Rbar^a_{\;\; rcr|r}\phi^c_{\;\; b}
        +\Rbar^a_{\;\; rcr}\partial_t\phi^c_{\;\; b}, \\
      &\partial_t^4\phi^a_{\;\; b}(t)=\Rbar^a_{\;\; rcr|rr}\phi^c_{\;\; b}
        +2\Rbar^a_{\;\; rcr|r}\partial_t\phi^c_{\;\; b}
        +\Rbar^a_{\;\; rcr}\partial_t^2\phi^c_{\;\; b}, \\
      &\partial_t^5\phi^a_{\;\; b}(t)=\Rbar^a_{\;\; rcr|rrr}\phi^c_{\;\; b}
        +3\Rbar^a_{\;\; rcr|rr}\partial_t\phi^c_{\;\; b}
        +3\Rbar^a_{\;\; rcr|r}\partial_t^2\phi^c_{\;\; b}
        +\Rbar^a_{\;\; rcr}\partial_t^3\phi^c_{\;\; b}, \mbox{ and } \\
      &\partial_t^6\phi^a_{\;\; b}(t)=\Rbar^a_{\;\; rc|rrrr}\phi^c_{\;\; b}
        +4\Rbar^a_{\;\; rcr|rrr}\partial_t\phi^c_{\;\; b}
        +6\Rbar^a_{\;\; rcr|rr}\partial_t^2\phi^c_{\;\; b}
        +4\Rbar^a_{\;\; rcr|r}\partial_t^3\phi^c_{\;\; b}
        +\Rbar^a_{\;\; rcr}\partial_t^4\phi^c_{\;\; b},
    \end{aligned}
  \end{equation}
so
  \begin{equation}
    \begin{aligned}
      \phi^a_{\;\; b}(t)
      =
      t\delta^a_{\;\; b}
        + \frac{1}{6}t^3\Rbar^a_{\;\; rbr}(0)
        + \frac{1}{12}t^4\Rbar^a_{\;\; rbr|r}(0)
        + \frac{1}{40}t^5\Rbar^a_{\;\; rbr|rr}(0)
        + \frac{1}{120}t^5R^a_{\;\; rcr}\Rbar^c_{\;\; rbr}(0) \\
          - \frac{1}{5!}\int_0^1 (t-1)^5 \partial_t^6\phi^a_{\;\; b}(t) \, dt.
    \end{aligned}
  \end{equation}

If we write $\gamma(t)$, $A(t)$, and $H(t)$
for the pullbacks to $S^n$ (identified with $\{t\} \times S^n$)
under $\Phi|_{\{t\} \times S^n}$
of, respectively, the metric, second fundamental form, and mean curvature
induced by $g$ on the metric sphere of radius $t$,
then $\gamma(t)(V,W)|_p=g(\phi V, \phi W)|_p$ and for small $t$ we obtain
  \begin{equation}
  \begin{aligned}
    &\norm{t^{-2}\gamma_{\mu \nu}(t) - 
      \left[
        \iota^*\delta_{\mu \nu} + \frac{1}{3}t^2 R_{\mu r \nu r} 
          + \frac{1}{6}t^3R_{\mu r \nu r | r}
          + t^4\left(\frac{1}{20}R_{\mu r \nu r | rr} 
              + \frac{2}{45}R_{\mu r \lambda r}R^\lambda_{\;\; r \nu r}\right)
      \right]}_{C^k} \leq C(k)t^5, \\
    &\norm{t^2\gamma^{\mu \nu}(t) - 
      \left[
        \iota^*\delta^{\mu \nu} - \frac{1}{3}t^2 R^{\mu \;\; \nu}_{\;\; r \;\; r} 
          - \frac{1}{6}t^3R^{\mu \;\; \nu}_{\;\; r \;\; r | r}
          - t^4\left(\frac{1}{20}R^{\mu \;\; \nu}_{\;\; r \;\; r | rr} 
              - \frac{1}{15}R^\mu_{\;\; r \lambda r}R^{\lambda \;\; \nu}_{\;\; r \;\; r}\right)
      \right]}_{C^k} \leq C(k)t^5, \\
    &\norm{t^{-1}A(t) - 
      \left[
        -\iota^*\delta_{\mu \nu} - \frac{2}{3}t^2 R_{\mu r \nu r} - \frac{5}{12}t^3R_{\mu r \nu r | r}
          - t^4\left(\frac{3}{20}R_{\mu r \nu r | rr} 
              + \frac{2}{15}R_{\mu r \lambda r}R^\lambda_{\;\; r \nu r}\right)
      \right]}_{C^k} \leq C(k)t^5, \mbox{ and} \\
    &\norm{tH(t) - 
      \left[
        -2 + \frac{1}{3}t^2R_{rr} + \frac{1}{4}t^3R_{rr|r}
          + t^4\left(\frac{1}{10}R_{rr|rr} + \frac{1}{45}R_{\mu r \nu r}R^{\mu r \nu r}\right)
      \right]}_{C^k} \leq C(k)t^5,
  \end{aligned}
  \end{equation}
where the curvature factors are all evaluated at $p$
(and we observe the same index-ordering convention as for $\overline{R}$ above)
and
where the constant $C(k)$ depends on just the suprema
of the norms of the curvature of $g$ and its first $k+4$ derivatives
on any closed geodesic ball in $M$ with center $p$ and radius at least $t$.

\end{document}